\documentclass[10pt]{amsart}
\usepackage{amsbsy,amssymb,amscd,amsfonts,latexsym,amstext,delarray,
amsmath,amsthm,graphicx}

\input xypic

\newtheorem{thm}{Theorem}[section]
\newtheorem{prop}[thm]{Proposition}
\newtheorem{cor}[thm]{Corollary}
\newtheorem{lem}[thm]{Lemma}

\newtheorem{claim}[thm]{Claim}
\newtheorem{defn}[thm]{Definition}
\newtheorem{rem}[thm]{Remark}
\newtheorem{ex}[thm]{Example}

\numberwithin{equation}{section}

\def\bG{{\mathbb G}}

\def\bL{{\mathbb L}}

\def\bT{{\mathbb T}}

\def\A{{\mathbb A}}
\def\C{{\mathbb C}}
\def\F{{\mathbb F}}

\def\N{{\mathbb N}}
\renewcommand{\P}{{\mathbb P}}
\renewcommand{\L}{{\mathbb L}}
\def\Q{{\mathbb Q}}
\def\Z{{\mathbb Z}}
\def\R{{\mathbb R}}
\def\K{{\mathbb K}}
\def\m{{\mathfrak{m}}}

\def\cD{{\mathcal D}}

\def\cL{{\mathcal L}}
\def\cM{{\mathcal M}}

\def\cO{{\mathcal O}}
\def\cP{{\mathcal P}}

\def\cT{{\mathcal T}}

\def\cV{{\mathcal V}}
\def\cW{{\mathcal W}}

\newcommand{\ie}{{\it i.e.\/}\ }
\newcommand{\eg}{{\it e.g.\/}\ }
\newcommand{\cf}{{\it cf.\/}\ }
\def\text{\hbox}

\def\SL{{\rm SL}}

\title[Feynman integrals and determinants]{Parametric Feynman 
integrals and determinant hypersurfaces}
\author[Aluffi]{Paolo Aluffi}
\author[Marcolli]{Matilde Marcolli}
\address{Department of Mathematics \\ Florida State University \\
Tallahassee, FL 32306, USA}
\email{aluffi\@@math.fsu.edu}
\address{Department of Mathematics \\ Caltech \\ Pasadena, CA 91125, USA}
\email{matilde\@@caltech.edu}
\begin{document}

\maketitle

\begin{abstract}
The purpose of this note is to show that, under certain combinatorial conditions 
on the graph, parametric Feynman integrals can be realized as periods on the 
complement of the determinant hypersurface $\hat\cD_\ell$ in affine space 
$\A^{\ell^2}$, 
with $\ell$ the number of loops of the Feynman graph. The question of whether
these are periods of mixed Tate motives can then be reformulated as a question
on a relative cohomology of the pair 
$(\A^{\ell^2}\smallsetminus \hat\cD_\ell, \hat\Sigma_{\ell,g}
\smallsetminus (\hat\Sigma_{\ell,g} \cap \hat\cD_\ell))$ being a realization of a 
mixed Tate motive, 
where $\hat\Sigma_{\ell,g} $ is a normal crossings divisor depending only on the number
of loops and the genus of the graph. We show explicitly that the relative cohomology is
a realization of a mixed Tate motive in the case of three loops and we give alternative 
formulations of the main question in the general case, by describing the locus 
$\hat\Sigma_{\ell,g} \smallsetminus (\hat\Sigma_\ell \cap \hat\cD_\ell)$  
in terms of intersections of unions of Schubert cells in flag varieties. 
We also discuss different methods of regularization aimed at removing the 
divergences of the Feynman integral. 
\end{abstract}

\section{Introduction}\label{intro}

The question of whether Feynman integrals arising in perturbative scalar quantum field
theory are periods of mixed Tate motives can be seen (see \cite{BEK}, \cite{Blo}) as 
a question on whether certain relative cohomologies associated to algebraic varieties
defined by the data of the parametric representation of the Feynman integral are 
realizations of mixed Tate motives. In this paper we investigate another possible
viewpoint on the problem, which leads us to consider a different relative cohomology,
defined in terms of the complement of the affine determinant hypersurface and the locus
where the hypersurface intersects the image of a simplex under a linear map defined 
by the Feynman graph.
For all graphs with a given number of loops~$\ell$, admitting a minimal
embedding in an orientable surface of genus~$g$, and satisfying a natural 
combinatorial condition, we relate the question mentioned above to a
problem in the geometry of coordinate subspaces of an $\ell$-dimensional
vector space, which only depends on the genus~$g$. 

More precisely, we consider for each graph $\Gamma$ as above and satisfying 
a transparent combinatorial condition (summarized at the beginning 
of~\S\ref{PeriodSec}) a normal 
crossing divisor $\hat\Sigma_\Gamma$ in the affine space $\A^{\ell^2}$ of 
$\ell\times\ell$ matrices. We observe that, modulo the issue of divergences, 
the parametric Feynman integral is a period of the pair 
$(\A^{\ell^2}\smallsetminus \hat\cD_\ell,
\hat\Sigma_\Gamma \smallsetminus (\hat\cD_\ell \cap \hat\Sigma_\Gamma))$,
where $\hat\cD_\ell$ is the determinant hypersurface. We then observe that
all these normal crossing divisors $\hat\Sigma_\Gamma$ may be immersed 
into a fixed normal crossing divisor $\hat\Sigma_{\ell,g}$, determined by
the number of loops $\ell$ and the embedding genus $g$; therefore, 
the question of whether Feynman integrals are
periods of mixed Tate motives may be decided by verifying that the
motive
$$ \m(\A^{\ell^2}\smallsetminus \hat\cD_\ell,
\hat\Sigma_\Gamma \smallsetminus (\hat\cD_\ell \cap \hat\Sigma_{\ell,g})), $$
whose realization is the relative cohomology of the corresponding pair, 
is mixed Tate.
In fact, we show that verifying this assertion for $g=0$ would suffices to deal with 
all graphs $\Gamma$ with $b_1(\Gamma)=\ell$ (and satisfying our combinatorial
condition), simultaneously for all genera.

We approach this question by an inclusion-exclusion argument, reducing it
to verifying that specific loci in $\A^{\ell^2}$ are mixed Tate (see \S\ref{Tmq}).
We carry out this verification for $\ell\le 3$ loops (\S\ref{MotFrameSec}), 
showing that the motive 
$\m(\A^{9}\smallsetminus \hat\cD_3, \hat\Sigma_{3,0} \smallsetminus 
(\hat\cD_3 \cap \hat\Sigma_{3,0})$ is mixed Tate.
In doing so, we obtain explicit formulae for
the class $[\hat\Sigma_{3,0} \smallsetminus (\hat\cD_3 \cap \hat\Sigma_{3,0})]$
(corresponding to the `wheel with three spokes')
and for the classes of strata of the same locus, in the Grothendieck group
of varieties. These classes may be assembled to construct the corresponding
class for any graph with three loops (satisfying our combinatorial condition).
This illustrates a simple case of our strategy: it follows that, modulo the issue
of divergences, Feynman integrals of graphs with three or fewer loops are 
indeed periods of mixed Tate motives. Carrying out the same strategy for a 
larger number of loops is a worthwhile project.

Finally, in \S\ref{Divs} we discuss the problem of regularization of divergent Feynman
integrals, and how different possible regularizations can be made compatible with
the approach via determinant hypersurfaces described here.

\bigskip

We recall the basic notation and terminology we use in the following.

\begin{defn}\label{FeyDiag}
Consider a scalar field theory with Lagrangian
\begin{equation}\label{Lagr}
 \cL(\phi)= \frac 12 (\partial \phi)^2 - \frac{m^2}{2} \phi^2 -
\cL_{int}(\phi),
\end{equation}
where $\cL_{int}(\phi)$ is a polynomial in $\phi$ of degree at least
three. Then a one particle irreducible (1PI) Feynman graph $\Gamma$ of the
theory is a finite connected graph with the following properties. 
\begin{itemize}
\item The valence of
each vertex is equal to the degree of one of the monomials in the
Lagrangian \eqref{Lagr}.
\item The set $E(\Gamma)$ of edges of the graph is divided into {\em
internal} and {\em external} edges, $E(\Gamma)=E_{int}(\Gamma)\cup
E_{ext}(\Gamma)$. Each internal edge connects two vertices of the
graph, while the external edges have only one vertex. (One thinks of
an internal edges as being a union of two half-edges and an external 
one as being a single half-edge.) 
\item The graph cannot be disconnected by removing a single internal
edge. This is the 1PI condition.
\end{itemize}
\end{defn}

In the following we denote by $n=\# E_{int}(\Gamma)$ the number of
internal edges, by $N=\# E_{ext}(\Gamma)$ the number of external
edges, and by $\ell=b_1(\Gamma)$ the number of loops.

In their parametric form, the Feynman integrals of {\em massless} perturbative
scalar quantum field theories (\cf \S 6-2-3 of \cite{ItZu}, 
\S 18 of \cite{BjDr}, and \S 6 of \cite{Naka})
are integrals of the form
\begin{equation}\label{paramInt}
U(\Gamma,p)=   \frac{\Gamma(n-D\ell/2)}{(4\pi)^{\ell D/2}} 
\int_{\sigma_n} \frac{P_\Gamma(t,p)^{-n+D\ell/2} \,\omega_n}
{\Psi_\Gamma(t)^{-n +(\ell+1)D/2}},
\end{equation}
where $\Gamma(n-D\ell/2)$ is a possibly divergent $\Gamma$-factor, 
$\sigma_n$ is the simplex
\begin{equation}\label{sigman}
\sigma_n=\{ (t_1,\ldots,t_n)\in \R_+^n \,|\, \sum_i t_i =1 \}
\end{equation}
and the polynomials $\Psi_\Gamma(t)$ and $P_\Gamma(t,p)$ are 
obtained from the combinatorics of the graph, respectively as 
\begin{equation}\label{PsiGamma}
\Psi_\Gamma(t)= \sum_{T\subset \Gamma} \prod_{e \notin E(T)} t_e,
\end{equation}
where the sum is over all the spanning trees $T$ of $\Gamma$ and
\begin{equation}\label{PGammapt}
P_\Gamma(p,t) = \sum_{C\subset \Gamma} s_C \prod_{e\in C} t_e, 
\end{equation}
where the sum is over the cut-sets $C\subset \Gamma$, \ie the
collections of $b_1(\Gamma)+1$ internal edges that divide the graph $\Gamma$ in
exactly two connected components $\Gamma_1\cup \Gamma_2$. 
The coefficient $s_C$ is a function of the external momenta attached 
to the vertices in either one of the two components
\begin{equation}\label{sCcoeff}
s_C = \left(\sum_{v\in V(\Gamma_1)} P_v\right)^2 = \left(\sum_{v\in
V(\Gamma_2)} P_v\right)^2.
\end{equation}
Here the $P_v$ are defined as
\begin{equation}\label{Pv}
P_v=\sum_{e\in E_{ext}(\Gamma), t(e)=v} p_e,
\end{equation}
where the $p_e$ are incoming external momenta attached to the external
edges of $\Gamma$ and satisfying the conservation law
\begin{equation}\label{extmom0}
 \sum_{e\in E_{ext}(\Gamma)} p_e =0. 
\end{equation}
In order to work with algebraic differential forms defined over $\Q$,
we assume that the external momenta are also taking rational
values $p_e\in \Q^D$.

Ignoring the $\Gamma$-function factor in \eqref{paramInt}, 
one is interested in understanding what kind of period is the
integral
\begin{equation}\label{Int}
\int_{\sigma_n} \frac{P_\Gamma(t,p)^{-n+D\ell/2} \,\omega_n}
{\Psi_\Gamma(t)^{-n +(\ell+1)D/2}}.
\end{equation} 

In quantum field theory one can consider the same physical theory (with specified
Lagrangian) in different spacetime dimensions $D\in \N$. In fact, one should think
of the dimension $D$ as one of the variable parameters in the problem. For the purposes
of this paper, we work in the range where $D$ is sufficiently large, so that $n\leq D\ell/2$.
The case $n= D\ell/2$ is the {\em log divergent} case, where the integral \eqref{Int} simplifies
to the form
\begin{equation}\label{logdiv}
\int_{\sigma_n} \frac{\omega_n}
{\Psi_\Gamma(t)^{D/2}}.
\end{equation}

Another case where the Feynman integral has the simpler form 
\eqref{logdiv}, even for graphs that do not necessarily satisfy
the log divergent condition, \ie for $n\neq D\ell/2$, is where
one considers the case with nonzero mass $m\neq 0$, but with
external momenta set equal to zero. In such cases, the parametric
Feynman integral becomes of the form
\begin{equation}\label{massFeyInt}
\int_{\sigma_n} \frac{V_\Gamma(t,p)^{-n+D\ell/2}\omega_n}
{\Psi_\Gamma(t)^{D/2}} |_{p=0}= m^{-2n+D\ell} \int_{\sigma_n} \frac{\omega_n}{\Psi_\Gamma(t)^{D/2}},
\end{equation}
where $V_\Gamma(t,p)$ is of the form
$$ V_\Gamma(t,p)= p^\dag R_\Gamma(t) p + m^2, $$
with 
$$ V_\Gamma(t,p)|_{m=0} = \frac{P_\Gamma(t,p)}{\Psi_\Gamma(t)}. $$

In the following we assume that we are either in the massless case \eqref{Int}
and in the range of dimensions $D$ satisfying $n\leq D\ell/2$, or in the massive
case with zero external momenta \eqref{massFeyInt} and arbitrary dimension.

A first issue one needs to clarify in addressing the question of Feynman integrals
and periods is the fact that the integral \eqref{Int} is often divergent. 
Divergences are
contributed by the intersection $\sigma_n \cap \hat X_\Gamma$, with
$\hat X_\Gamma =\{ t\in \A^n \,|\, \Psi_\Gamma(t)=0 \}$, which is often 
non-empty. Although there are cases where a nonempty intersection
$\sigma_n \cap \hat X_\Gamma$ may still give rise to an absolutely convergent
integral, hence a period, these are relatively rare cases and usually some
regularization and renormalization procedure is needed to eliminate the
divergences over the locus where the domain of integration meets the
graph hypersurface. Notice that these intersections only occur on the
boundary $\partial \sigma_n$, since in the interior of $\sigma_n$ the
polynomial $\Psi_\Gamma(t)$ is strictly positive (see \eqref{PsiGamma}).

Our results will apply directly to all cases where the integral is convergent, 
while we discuss in Section \ref{Divs} the case where a regularization procedure
is required to treat divergences in the Feynman integrals. 
The main question is then, more precisely formulated, whether it is 
true that the numbers obtained by computing such integrals 
(after removing a possibly divergent Gamma factor, and after regularization and
%
renormalization when needed) are always periods of mixed Tate motives.

The main contribution of this paper is the reformulation of the problem,
where instead of working with the graph hypersurfaces $X_\Gamma$ defined 
by the vanishing of the graph polynomial $\Psi_\Gamma$, one works with the complement of a fixed determinant hypersurface in an affine space of matrices.
This allows us to reduce the problem to one that only depends on the number
of loops of the graph, at least for the class of graphs satisfying the 
combinatorial condition discussed in \S\ref{DetSec} (for example, 
$3$-vertex connected planar graphs with $\ell$ loops). We propose
specific questions in terms of $\ell$ alone, in \S\ref{Tmq}; these questions 
may be appreciated independently of our motivation, as they do not refer
directly to Feynman graphs.
We hope that these reformulations
might help to connect the problem to other interesting
questions, such as the geometry of intersections of Schubert cells and
Kazhdan--Lusztig theory.

\section{Feynman parameters and determinants}\label{DetSec}

With the notation as above, 
for a given Feynman graph $\Gamma$, the graph hypersurface $X_\Gamma$
is defined as the locus of zeros
\begin{equation}\label{XGammaProj}
X_\Gamma =\{ t=(t_1:\ldots :t_n)\in \P^{n-1} \,| \, \Psi_\Gamma(t)=0 \}.
\end{equation}
Indeed, $\Psi_\Gamma$ is homogeneous of degree $\ell$, hence 
it defines a hypersurface of degree $\ell$ in the projective space $\P^{n-1}$.
We will also consider the affine cone on $X_\Gamma$, namely the affine
hypersurface
\begin{equation}\label{hatXGamma}
\hat X_\Gamma =\{ t\in \A^n \,|\, \Psi_\Gamma(t)=0 \}.
\end{equation}

The question of whether the Feynman integral is a period of a mixed Tate 
motive can be approached (modulo the divergence problem) as a question
on whether the relative cohomology
\begin{equation}\label{relcohomXGamma}
H^{n-1}(\P^{n-1}\smallsetminus X_\Gamma, \Sigma_n \smallsetminus (\Sigma_n\cap X_\Gamma))
\end{equation}
is a realization of a mixed Tate motive, where $\Sigma_n$ is the algebraic simplex 
\begin{equation}\label{Sigman}
\Sigma_n =\{ t \in \P^{n-1} \,|\, \prod_i t_i =0 \},
\end{equation}
\ie the union of the coordinate hyperplanes containing the boundary of the domain of
integration $\partial \sigma_n \subset \Sigma_n$. See for instance \cite{BEK}, \cite{Blo}.

Although working in the projective setting is very natural (see \cite{BEK}), there are 
several reasons why it may be preferable to consider affine hypersurfaces:
\begin{itemize}
\item Only in the limit cases of a massless theory or of zero external momenta in the
massive case does the parameteric Feynman integral involve the quotient of two homogeneous
polynomial (\cite{BjDr}, \S 18).
\item The deformations of the $\phi^4$ quantum field theory to noncommutative spacetime,
which has been the focus of much recent research (see \eg \cite{GrWu}), 
shows that, even in the massless case
the graph polynomials $\Psi_\Gamma$ and $P_\Gamma$ are no longer homogeneous in
the noncommutative setting and only in the limit commutative case they recover this property (see 
\cite{GuRi}, \cite{KrRiTaWa}).
\item As shown in \cite{AluMa2}, in the affine setting the graph hypersurface complement
satisfies a multiplicative property over disjoint unions of graphs that makes it possible to
define algebro-geometric and motivic Feynman rules. 
\end{itemize}
For these various reasons, in this paper we primarily work in the affine rather than in 
the projective setting.

In the present paper, we approach the problem in a different way, where instead of
working with the hypersurface $\hat X_\Gamma$, we map the 
Feynman integral computation and the graph hypersurface in a larger hypersurface 
$\hat \cD_\ell$ inside a larger affine space, so that we will be dealing with a relative 
cohomology replacing \eqref{relcohomXGamma} where the ambient space (the
hypersurface complement) only depends on the number of loops in the graph.

\subsection{Determinant hypersurfaces and graph polynomials}\label{Dhagp}

We now show that all the affine varieties $\hat X_\Gamma$, for fixed 
number of loops $\ell$, map naturally to a larger hypersurface in a 
larger affine space, by realizing the polynomial $\Psi_\Gamma$ for the given
graph as a pullback of a fixed polynomial $\Psi_\ell$ in $\ell^2$-variables.

Recall that the determinant hypersurface $\cD_\ell$ is defined in the
following way. Let $k[x_{kr}, k,r=1,\ldots,\ell]$ be the polynomial
ring in $\ell^2$ variables and set
\begin{equation}\label{Detell}
\cD_\ell = \{ x=(x_{kr})\,|\, \det(x)=0 \}.
\end{equation}
Since the determinant is a homogeneous polynomial $\Psi_\ell$, this in particular
also defines a projective hypersurface in $\P^{\ell^2 -1}$.  We will however mostly
concentrate on the affine hypersurface $\hat\cD_\ell\subset \A^{\ell^2}$ 
defined by the vanishing of the determinant, \ie the cone in $\A^{\ell^2}$ 
of the projective hypersurface $\cD_\ell$.

Suppose given any Feynman graph $\Gamma$ with
$b_1(\Gamma)=\ell$, and with $\# E_{int}(\Gamma)=n$. 
It is well known (see \eg \S 18 of \cite{BjDr}) that the graph
polynomial $\Psi_\Gamma(t)$ can be equivalently written in the form of
a determinant
\begin{equation}\label{PsiDet}
\Psi_\Gamma(t)=\det M_\Gamma(t)
\end{equation}
of an $\ell\times \ell$-matrix
\begin{equation}\label{MGamma}
(M_\Gamma)_{kr}(t)=\sum_{i=1}^n t_i \eta_{ik} \eta_{ir},
\end{equation} 
where the $n \times \ell$-matrix $\eta_{ik}$ is defined in terms of
the edges $e_i \in E(\Gamma)$ and   
a choice of a basis for the first homology group, $l_k \in
H_1(\Gamma,\Z)$,  with $k=1,\ldots, \ell=b_1(\Gamma)$, by setting 
\begin{equation}\label{etaik}
\eta_{ik}=\left\{ \begin{array}{rl} +1 & \text{edge $e_i\in$ loop
$l_k$, same orientation} \\[2mm] -1 & \text{edge $e_i\in$ loop
$l_k$, reverse orientation} \\[2mm] 0 & \text{otherwise.} \end{array}\right.
\end{equation}
The determinant $\det M_\Gamma(t)$ is independent both of the choice of
orientation on the edges of the graph and of the choice of generators for
$H_1(\Gamma,\Z)$.

The expression of the matrix $M_\Gamma(t)$ defines a linear map  $\tau:\A^n \to 
\A^{\ell^2}$ of the form 
\begin{equation}\label{xmap}
\tau=\tau_\Gamma: \A^n \to \A^{\ell^2}, \ \ \ \ \tau(t_1,\ldots,t_n)= 
\sum_i t_i \eta_{ki}\eta_{ir}. 
\end{equation}
We can write this equivalently in the shorter form 
\begin{equation}\label{taumap}
\tau = \eta^\dag \Lambda \eta,
\end{equation}
where $\Lambda$ is the diagonal $n\times n$-matrix with $t_1,\ldots,t_n$ 
as diagonal entries, and $\eta=\eta_\Gamma$ is the matrix \eqref{etaik}.

Then by construction we have that $\hat X_\Gamma = \tau^{-1}(\hat \cD_\ell)$,
from \eqref{PsiDet}.
We formalize this as follows:

\begin{lem}\label{XGammaDell}
Let $\Gamma$ be a Feynman graph with $n$ internal edges and $\ell$ loops.
Let $\hat X_\Gamma \subset \A^n$ denote the affine cone on the projective 
hypersurface $X_\Gamma \subset \P^{n-1}$. Then 
\begin{equation}\label{hatXGammaDell}
\hat X_\Gamma =\tau^{-1}(\hat \cD_\ell),
\end{equation}
where $\tau: \A^n \to \A^{\ell^2}$ is a linear map depending on $\Gamma$.
\end{lem}

The next lemma, which follows directly from the definitions, details some
of the properties of the map $\tau$ introduced above that we will be using 
in the following.

\begin{lem}\label{mapstaui}
The matrix of $\tau$, $M_\Gamma(t)=\eta^\dag \Lambda \eta$,
has the following properties.
\begin{itemize}
\item For $i\ne j$, the corresponding entry is the sum of
$\pm t_k$, where the $t_k$ correspond to the edges common to
the $i$-th and $j$-th loop, and the sign is $+1$ if the orientations
of the edges both agree or both disagree with the loop orientations,
and $-1$ otherwise.
\item For $i=j$, the entry is the sum of the variables $t_k$ corresponding
to the edges in the $i$-th loop (all taken with sign~$+$).
\end{itemize}
Now consider a specific edge $e$, and let $t_e$ be the corresponding 
variable. Then
\begin{itemize}
\item The variable $t_e$ appears in $\eta^\dag \Lambda\eta$ 
if and only if $e$ is part of at least one loop.
\item If $e$ belongs to a single loop $\ell_i$, then $t_e$ only appears in the
diagonal entry $(i,i)$, added to the variables corresponding to the other
edges forming the loop $\ell_i$.
\item If there are two loops $\ell_i$, $\ell_j$ containing $e$, and not having
any other edge in common, then the $\pm t_e$ appears by itself at the entries
$(i,j)$ and $(j,i)$ in the matrix $\eta^\dag \Lambda\eta$.
\end{itemize}
\end{lem}

When the map $\tau$ constructed above is injective, it is possible to rephrase 
the computation of the parametric Feynman integral \eqref{Int} as a period of the
complement of the determinant hypersurface $\hat \cD_\ell\subset \A^{\ell^2}$.

\begin{lem}\label{periodDell}
Assume that the map $\tau: \A^n \to \A^{\ell^2}$ of \eqref{taumap} is injective. 
Then the integral \eqref{Int} can be rewritten in the form
\begin{equation}\label{intDell}
\int_{\tau(\sigma_n)} \frac{\cP_\Gamma(p,x)^{-n+D\ell/2} 
\,\omega_\Gamma(x)}{\det(x)^{-n +(\ell+1)D/2}},
\end{equation}
where $\cP_\Gamma(p,x)$ is a homogeneous polynomial on $\A^{\ell^2}$ whose
restriction to the image of $\A^n$ under the map $\tau$
agrees with $P_\Gamma(p,t)$, and $\omega_\Gamma$ is the induced volume form.
\end{lem}

\proof It is possible to regard the
polynomial $P_\Gamma(p,t)$ as the restriction to $\A^n$ of a homogeneous
polynomial $\cP_\Gamma(p,x)$ defined on all of
$\A^{\ell^2}$. Clearly, such $\cP_\Gamma(p,x)$ will not be unique,
but different choices of $\cP_\Gamma(p,x)$ will not affect the
integral calculation, which all happens inside the linear subspace
$\A^n$. The simplex $\sigma_n$ is also linearly embedded inside 
$\A^{\ell^2}$, and we denote its image by
$\tau(\sigma_n)$. The volume form $\omega_n$ can also be
identified, under such a choice of coordinates in $\A^{\ell^2}$ 
with a form $\omega_\Gamma(x)$ such that
$$ \omega_\Gamma(x)\wedge \langle \xi_\Gamma, dx \rangle =
\omega_{\ell^2}, $$
with $\xi_\Gamma$ the $(\ell^2-n)$-frame associated to the linear
subspace $\tau(\A^n)\subset \A^{\ell^2}$ and 
$$ \langle \xi_\Gamma, dx \rangle = \langle \xi_1, dx
\rangle\wedge \cdots \wedge \langle \xi_{\ell^2-n}, dx \rangle. $$
\endproof

Notice in particular that if the map $\tau$ is injective then one has a well defined
map $\P^{n-1} \to \P^{\ell^2-1}$, which is otherwise not everywhere
defined. 

We are interested in the following, heuristically formulated, consequence of Lemma~\ref{periodDell}.

\begin{claim}\label{pairmot}
Assume that the map $\tau: \A^n \to \A^{\ell^2}$ of \eqref{taumap} is injective. 
Then the complexity of Feynman integrals corresponding to the graph $\Gamma$
is controlled by the motive $\m(\A^{\ell^2}\smallsetminus 
\hat\cD_\ell, \hat\Sigma_\Gamma \smallsetminus (\hat\cD_\ell\cap \hat\Sigma_\Gamma))$,
where $\hat\Sigma_\Gamma$ is a normal crossings divisor in $\A^{\ell^2}$
such that $\tau(\partial \sigma_n)\subset \hat\Sigma_\Gamma$.
\end{claim}

The explicit construction of the normal crossings divisor $\hat\Sigma_\Gamma$ is given in
Lemma \ref{tauinjXi} below.
We will further improve on this observation by reformulating it in a way that
will only depend on the number of loops $\ell$ of $\Gamma$ and on its
genus, and not on
the specific graph $\Gamma$. To this purpose, we will determine subsets of $\A^{\ell^2}$
which will contain the components of the image $\tau(\partial\sigma_n)$ of the
boundary of the simplex in $\A^n$, independently of $\Gamma$ (see \S \ref{DoG}).

In any case, this type of results motivates us to determine
conditions on the Feynman graph $\Gamma$ which
ensure that the corresponding map $\tau: \A^n \to \A^{\ell^2}$ is injective.

\section{Graph theoretic conditions for embeddings}\label{Grcfe}

\subsection{Injectivity of $\tau$}
In the following, we denote by $\tau_i$ the composition of the map $\tau$
of \eqref{taumap} with the projection to the $i$-th row of the matrix 
$\eta^\dag \Lambda \eta$, viewed as
a map of the variables corresponding only to the edges that belong 
to the $i$-th loop in the chosen bases of the first homology of the graph $\Gamma$.

We first make the following simple observation. 

\begin{lem}\label{tauitauinj}
If $\tau_i$ is injective for $i$ ranging over a set of loops such that every
edge of $\Gamma$ is part of a loop in that set, then $\tau$ is itself injective.
\end{lem}

\proof Let $(t_1,\dots,t_n)=(c_1,\dots,c_n)$ be in the kernel of $\tau$. Since 
each $(i,j)$ entry in the target matrix is a combination of edges in the 
$i$-th loop, the map $\tau_i$ must send to zero the tuple of $c_j$'s
corresponding to the edges in the $i$-th loop. Since we are assuming
$\tau_i$ to be injective, that tuple is the zero-tuple. Since every edge
is in some loop for which $\tau_i$ is injective, it follows that every $c_j$ 
is zero, as needed.
\endproof

The properties detailed in Lemma~\ref{mapstaui} immediately provide a 
sufficient condition for the maps $\tau_i$ to be injective.

\begin{lem}\label{injtaui}
The map $\tau_i$ is injective if the following conditions are satisfied:
\begin{itemize} 
\item For every edge $e$ of the $i$-th loop, there is another loop having 
only $e$ in common with the $i$-th loop, and
\item The $i$-th loop has at most one edge not in common with any
other loop.
\end{itemize}
\end{lem}

\proof  In this situation, all but at most one edge variable 
appear by themselves as an entry of the $i$-th row, and the possible last
remaining variable appears summed together with the other variables.
More explicitly, if $t_{i_1},\dots,t_{i_v}$ are the variables 
corresponding to the edges of a loop $\ell_i$, up to rearranging the entries in the
corresponding row of $\eta^\dag \Lambda\eta$ 
and neglecting other entries,  the map $\tau_i$ is given by
\[
(t_{i_1},\dots,t_{i_v}) \mapsto (t_{i_1}+\cdots+t_{i_v},
\pm t_{i_1},\dots,\pm t_{i_v})
\]
if $\ell_i$ has no edge not in common with any other loop, and 
\[
(t_{i_1},\dots,t_{i_v}) \mapsto (t_{i_1}+\cdots+t_{i_v},
\pm t_{i_1},\dots,\pm t_{i_{v-1}})
\]
if $\ell_i$ has a single edge $t_v$ not in common with any other loop.
In either case the map $\tau_i$ is injective, as claimed.
\endproof

Now we need a sufficiently natural combinatorial condition on the graph $\Gamma$ 
that ensures that the conditions of Lemma~\ref{injtaui} and 
Lemma~\ref{tauitauinj} are fulfilled. 
We first recall some useful facts about graphs and embeddings of graphs 
on surfaces which we need in the following.

Every (finite) graph $\Gamma$
may be embedded in a compact orientable surface of 
finite genus. The minimum genus of an orientable surface in which 
$\Gamma$ may be embedded is the {\em genus\/} of $\Gamma$. 
Thus, $\Gamma$ is planar if and only if it may be embedded in a sphere, 
if and only if its genus is~$0$. 

\begin{defn}\label{2cell}
An embedding of a graph $\Gamma$ in an orientable surface $S$ is 
a {\em 2-cell embedding} if the complement of $\Gamma$ in $S$
is homeomorphic to a union of open 2-cells (the {\em faces\/}, or {\em regions\/}
determined by the embedding). An embedding of $\Gamma$ in $S$ is a 
{\em closed 2-cell embedding} if the closure of every face is a disk.
\end{defn}

It is known that an embedding of a connected graph is minimal genus  if and only 
if it is a 2-cell embedding (\cite{MoTho}, Proposition~3.4.1 and Theorem~3.2.4). 
We discuss below conditions on the existence of {\em closed} 
2-cell embeddings, \cf \cite{MoTho}, \S 5.5.

For our purposes, the advantage of having a closed 
2-cell embedding for a graph $\Gamma$ is
that the faces of such an embedding determine a choice of loops of $\Gamma$, by
taking the boundaries of the 2-cells of the embedding together with a basis of generators 
for the homology of the Riemann surface in which the graph is embedded.

\begin{lem}\label{2cellH1}
A closed 2-cell embedding $\iota:\Gamma \to S$ of a connected graph $\Gamma$ 
on a surface of (minimal) genus $g$, together with the 
choice of a face of the embedding 
and a basis for the homology $H_1(S,\Z)$ determine
a basis of $H_1(\Gamma,\Z)$ given by $2g+f-1$ loops, 
where $f$ is the number of faces of the embedding.
\end{lem}

\proof
Orient (arbitrarily) the edges of $\Gamma$ and the faces, and 
then add the edges on the boundary
of each face with sign determined by the orientations. 
The fact that the closure of each face is a 
2-disk guarantees that the boundary is null-homotopic. 
This produces a number of loops equal to the number $f$ of faces. 
It is clear that these $f$ loops are {\em not\/} independent: the
sum of any $f-1$ of them must equal the remaining one, up to
sign. Any $f-1$ loops, however, will be independent
in $H_1(\Gamma)$. Indeed, these $f-1$ loops, together with $2g$
generators of the homology of $S$, generate $H_1(\Gamma)$. 
The homology group
$H_1(\Gamma)$ has rank $2g+f-1$, as one can see from 
the Euler characteristic formula
$$ b_0(S)-b_1(S)+b_2(S)=2-2g=\chi(S)=v-e+f=b_0(\Gamma)-b_1(\Gamma)+f =1-\ell +f, $$
so there will be no other relations.
\endproof

One refers to the chosen one among the $f$ faces as the ``external
face" and the remaining $f-1$ faces as the ``internal faces".

Thus, given a closed 2-cell embedding $\iota:\Gamma \to S$, 
we can use a basis of $H_1(\Gamma,\Z)$ costructed as in 
Lemma~\ref{2cellH1} to compute the map $\tau$ of \eqref{taumap} 
and the maps $\tau_i$ of \eqref{mapstaui}. We then have the 
following result.

\begin{lem}\label{expli1}
Assume that $\Gamma$ is closed-$2$-cell embedded in a surface.
With notation as above, assume that 
\begin{itemize}
\item any two of the $f$ faces have at most one edge in common.
\end{itemize}
Then the $f-1$ maps $\tau_i$, defined with respect to a choice of basis for
$H_1(\Gamma)$ as in Lemma~\ref{2cellH1}, are all injective.
If further
\begin{itemize}
\item every edge of $\Gamma$ is in the boundary of two of the $f$ faces, 
\end{itemize}
then $\tau$ is injective.
\end{lem}

\begin{proof}
The injectivity of the $f-1$ maps $\tau_i$ follows from Lemma~\ref{injtaui}.
If $\ell$ is a loop determined by an internal
face, the variables corresponding to edges in common between $\ell$
and any other internal loop will appear as ($\pm$) individual entries
on the row corresponding to $\ell$. Since $\ell$ has at most one
edge in common with the external region, this accounts for all but
at most one of the edges in $\ell$. By Lemma~\ref{injtaui},
the injectivity of $\tau_i$ follows. 

Finally, as shown in Lemma~\ref{tauitauinj}, the map $\tau$ is injective if 
every edge is in one of the $f-1$ loops and the $f-1$ maps $\tau_i$ are
injective. The stated condition guarantees that the edge appears in the 
loops corresponding to the faces separated by that edge. At least
one of them is internal, so that every edge is accounted for.
\end{proof}

\begin{ex}\label{graphexample}{\rm 
Consider the example of the planar graph in Figure \ref{planarex}. 
The conditions stated in Lemma~\ref{expli1} are evidently satisfied.
Edges are marked by circled numbers. The loop corresponding to
region~1 consists of edges 1, 2, 3, 4. The corresponding row of 
$\eta^t T\eta$ is
\[
(t_1+t_2+t_3+t_4,\pm t_4, \pm t_3, \pm t_2, \pm t_1)\quad.
\]
Region~2 consists of edges 4,5,6,7. Edge 7 is not in any other
internal region. The corresponding row of $\eta^\dag \Lambda \eta$ is
\[
(t_4+t_5+t_6+t_7,\pm t_4, \pm t_5, \pm t_6)\quad.
\]
These maps are injective, as claimed. Given the symmetry of
the situation, it is clear that all maps $\tau_i$ (and hence $\tau$
as well) are injective for this graph, as guaranteed by 
Lemma~\ref{expli1}.}\end{ex}

\begin{center}
\begin{figure}
\includegraphics[scale=0.9]{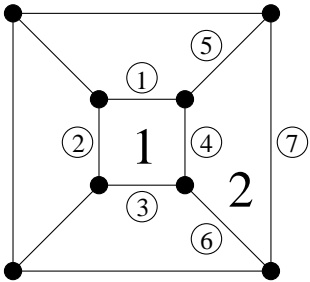}
\caption{An example satisfying Lemma~\ref{expli1}.
\label{planarex}}
\end{figure}
\end{center}

The considerations that follow will
allow us to improve on Lemma~\ref{expli1}, by showing 
that in natural situations the second condition listed in Lemma~\ref{expli1} is 
automatically satisfied.

\subsection{Connectivity of graphs}

In this section we review some notions on connectivity for graphs, both for
contextual reasons, since these notions relate well with conditions that are
natural from the physical point of view, and in order to improve the results
obtained above.

Given a graph $\Gamma$ and a vertex $v\in V(\Gamma)$, the graph 
$\Gamma \smallsetminus v$
is the graph with vertex set $V(\Gamma)\smallsetminus \{ v\}$ and edge set 
$E(\Gamma)\smallsetminus \{e:\, v\in \partial(e)\}$, \ie the graph 
obtained by removing from 
$\Gamma$ the star of the vertex $v$. 
It is customary to refer to $\Gamma \smallsetminus v$ simply 
as ``the graph obtained by removing the vertex $v$", even though 
one in fact removes also all the edges adjacent to~$v$.

There are two different notions of connectivity for graphs. To avoid confusion, we
refer to them here as $k$-edge-connectivity and $k$-vertex-connectivity. 
For the notion of $k$-vertex connectivity we follow \cite{MoTho} p.11, though 
in our notation graphs include the case of multigraphs.

\begin{defn}\label{kconnected}
The notions of $k$-edge-connectivity and $k$-vertex-connectivity are defined 
as follows:
\begin{itemize}
\item A graph is $k$-edge-connected if it cannot be disconnected by removal of 
any set of $k-1$ (or fewer) edges.
\item A graph is $2$-vertex-connected if it has no looping edges, it has at least
$3$ vertices, and it cannot be disconnected by removal of a single vertex, where
vertex removal is defined as above. 
\item For $k\geq 3$, a graph is $k$-vertex-connected if it has no looping
edges and no multiple edges, it has at least $k+1$ vertices, and it 
cannot be disconnected by removal of any set of $k-1$ vertices.
\end{itemize}
\end{defn}

Thus, $1$-vertex-connected and $1$-edge-connected simply mean 
connected, while $2$-edge-connected is the one-particle-irreducible
 (1PI) condition recalled in Definition~\ref{FeyDiag}.
To see how the condition of 
2-vertex-connectivity relates to the physical 1PI condition, we first
recall the notion of {\em splitting of a vertex}
in a graph $\Gamma$ (\cf \cite{MoTho}, \S 4.2).

\begin{defn}\label{splitvertex}
A graph $\Gamma'$ is a splitting of $\Gamma$ at a vertex $v\in V(\Gamma)$ if
it is obtained by partitioning the set $E\subset E(\Gamma)$ of edges adjacent to $v$ 
into two disjoint non-empty subsets, $E=E_1\cup E_2$ and inserting a new edge 
$e$ to whose end vertices $v_1$ and $v_2$ the edges in the two sets $E_1$ and $E_2$ 
are respectively attached (see Figure \ref{SplitVertFig}).
\end{defn}

\begin{center}
\begin{figure}
\includegraphics[scale=0.5]{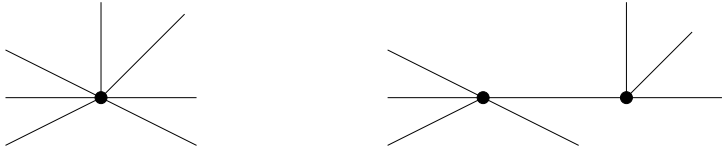}
\caption{A splitting of a graph $\Gamma$ at a vertex $v$
\label{SplitVertFig}}
\end{figure}
\end{center}

We have the following relation between $2$-vertex-connectivity and
$2$-edge-connectivity (1PI). 
The first observation will be needed in the proof of Proposition~\ref{2faces}; the
second is offered mostly for contextual reasons.

\begin{lem}\label{2edgevertex}
Let $\Gamma$ be a graph with at least $3$ vertices 
and no looping edges.
\begin{enumerate}
\item If $\Gamma$ is $2$-vertex-connected then it is also $2$-edge-connected (1PI).
\item  $\Gamma$ is $2$-vertex-connected if and only if all the graphs $\Gamma'$ obtained as 
splittings of $\Gamma$ at any $v\in V(\Gamma)$ are $2$-edge-connected (1PI).
\end{enumerate}
\end{lem}

\proof   (1): We have to show that, for a graph $\Gamma$ with at least $3$ vertices
and no looping edges, $2$-vertex-connectivity 
implies $2$-edge-connectivity.
Assume that $\Gamma$ is not 1PI. Then there exists an edge $e$ such
that $\Gamma\smallsetminus e$ has two connected components $\Gamma_1$
and $\Gamma_2$. Since $\Gamma$ has no looping edges, $e$ has two distinct
endpoints $v_1$ and $v_2$, which belong to the two different components
after the edge removal. 
Since $\Gamma$ has at least $3$ vertices, at least
one of the two components contains at least two vertices. Assume then that
there exists $v\neq v_1 \in V(\Gamma_1)$. Then, after the removal of the vertex 
$v_1$ from $\Gamma$, the vertices $v$ and $v_2$ belong to different connected
components, so that $\Gamma$ is not 2-vertex-connected.

(2): We need to show that $2$-vertex-connectivity is  
equivalent to all splittings $\Gamma'$ being 1PI.
Suppose first that $\Gamma$ is not $2$-vertex-connected. Since $\Gamma$ has at
least $3$ vertices and no looping edges, the failure of $2$-vertex-connectivity
means that there exists a vertex $v$ whose
removal disconnects the graph. Let $V\subset V(\Gamma)$ be the set 
of vertices other than $v$ that are endpoints of the edges adjacent to~$v$.  
This set is a union $V=V_1 \cup V_2$ where
the vertices in the two subsets $V_i$ are contained in at least two different 
connected components of $\Gamma\smallsetminus v$. Then the 
splitting $\Gamma'$ of $\Gamma$ at $v$ obtained by inserting 
an edge $e$ such that
the endpoints $v_1$ and $v_2$ are connected by edges, respectively, 
to the vertices in $V_1$ and $V_2$ is not 1PI. 

Conversely, assume that there exists a splitting $\Gamma'$ of $\Gamma$ 
at a vertex $v$ that is not 1PI. There exists an edge $e$ of $\Gamma'$ whose
removal disconnects the graph. If  $e$ already belonged to $\Gamma$,
then $\Gamma$ would not be 1PI (and hence not $2$-vertex connected, by (1)),
as removal of $e$  would disconnect it. So  $e$ must be the edge added in
the splitting of $\Gamma$ at the vertex $v$. 

Let $v_1$ and $v_2$ be the endpoints of $e$. None of the other edges
adjacent to $v_1$ or $v_2$ is a looping edge, by hypothesis; therefore
there exist at least another vertex $v'_1\ne v_2$ adjacent to $v_1$, and 
a vertex $v'_2\ne v_1$ adjacent to $v_2$. Since $\Gamma'\smallsetminus e$ 
is disconnected, $v'_1$ and $v'_2$ are in distinct connected components
of $\Gamma'\smallsetminus e$. Since $v'_1$ and $v'_2$ are in $\Gamma
\smallsetminus v$, and $\Gamma \smallsetminus v$ is contained in
$\Gamma' \smallsetminus e$, it follows that removing $v$ from $\Gamma$
would also disconnect the graph. Thus $\Gamma$ is not $2$-vertex-connected.
\endproof

The first statement in Lemma~\ref{2edgevertex} admits the following
analog for 3-connectivity.

\begin{lem}\label{3edgever}
Let $\Gamma$ be a graph with at least $4$ vertices, with no looping edges
and no multiple edges. Then $3$-vertex-connectivity implies 
$3$-edge-connectivity.
\end{lem}

\proof We argue by contradiction. 
Assume that $\Gamma$ is $3$-vertex-connected but not 2PI. We know it is 1PI because of
the previous lemma. Thus, there exist two edges $e_1$ and $e_2$
such that the removal of both edges is needed to disconnect the graph. Since
we are assuming that $\Gamma$ has no multiple or looping edges,
the two edges have at most one end in common. 

Suppose first
that they have a common endpoint $v$. Let $v_1$ and $v_2$ denote
the remaining two endpoints, $v_i\in \partial e_i$, $v_1\neq v_2$.  If
the vertices $v_1$ and $v_2$ belong to different connected components
after removing $e_1$ and $e_2$, then the removal of the vertex $v$ disconnects
the graph, so that $\Gamma$ is not 3-vertex-connected (in fact not even 
2-vertex-connected). If $v_1$ and~$v_2$ belong to the same connected 
component, then $v$ must be in a different component. 
Since the graph has at least $4$ vertices and no multiple or looping 
edges, there exists at least another edge attached to either $v_1$, $v_2$, 
or $v$, with the other endpoint $w\notin \{ v, v_1,v_2 \}$. 
If $w$ is adjacent to $v$, then removing $v$ and $v_1$ leaves $v_2$
and $w$ in different connected components. Similarly, if $w$ is
adjacent to (say) $v_1$, then the removal of the two vertices $v_1$ and
$v_2$ leave $v$ and $w$ in two different connected components.
Hence $\Gamma$ is not 3-vertex-connected. 

Next, suppose that
$e_1$ and $e_2$ have no endpoint in common. Let $v_1$ and $w_1$
be the endpoints of $e_1$ and $v_2$ and $w_2$ be the endpoints
of $e_2$. At least one pair $\{ v_i, w_i\}$ belongs to two separate
components after the removal of the two edges, though not all four
points can belong to different connected components, else the graph
would not be 1PI. Suppose then that $v_1$ and $w_1$ are in different
components. It also cannot happen that $v_2$ and $w_2$ belong to
the same component, else the removal of $e_1$ alone would
disconnect the graph. We can assume then that, say, $v_2$ belongs 
to the same component as $v_1$ while $w_2$ belongs to a
different component (which may or may not be the same as that of 
$w_1$). Then the removal of $v_1$ and $w_2$ leaves $v_2$ and $w_1$ in two
different components so that the graph is not 3-vertex-connected.
\endproof

\begin{rem}\label{Landau}
{\rm
While the 2-edge-connected hypothesis on Feynman graphs is very
natural from the physical point of view, since it is just the 1PI condition 
that arises when one considers the perturbative expansion of the
effective action of the quantum field theory (\cf~\cite{ItZu}), conditions
of 3-connectivity (3-vertex-connected or 3-edge-connected)
arise in a more subtle manner in the theory of Feynman integrals,
in the analysis of Laundau singularities (see for instance \cite{SMJO}).
In particular, the 2PI effective action is often considered in quantum
field theory in relation to non-equilibrium phenomena, see
\eg \cite{Rammer}, \S 10.5.1. }
\end{rem}

\subsection{Connectivity and embeddings}
We now recall another property of graphs on surfaces, namely 
the {\em face width} of an
embedding $\iota: \Gamma \hookrightarrow S$. 
The face width $fw(\Gamma,\iota)$ is
the largest number $k\in \N$ such that every 
non-contractible simple closed curve in $S$
intersects $\Gamma$ at least $k$ times. 
When $S$ is a sphere, hence $\iota:\Gamma
\hookrightarrow S$ is a planar embedding, 
one sets $fw(\Gamma,\iota)=\infty$.

\begin{rem}\label{closed2cell}
{\rm
For a graph $\Gamma$ with at least $3$ vertices and with no looping
edges, the condition that an embedding
$\iota: \Gamma \hookrightarrow S$ is a {\em closed} 2-cell embedding is
equivalent to the properties that $\Gamma$ is 2-vertex-connected and 
that the embedding has face width $fw(\Gamma,\iota)\geq 2$, 
see Proposition~5.5.11 of \cite{MoTho}.}
\end{rem}  

In particular, 
this implies that a planar graph with at least three vertices and no
looping edges admits a closed 2-cell embedding 
in the sphere if and only if it is 2-vertex-connected. Notice that the
condition that $\Gamma$ has at least $3$ vertices and no looping
edges is necessary for this statement to be true. For example, the graph 
with two vertices, one edge between them, and one looping edge attached 
to each vertex cannot be disconnected by removal of a single vertex,
but does not have a closed 2-cell embedding in the sphere. Similarly,
the graph consisting of two vertices, one edge between them and
one looping edge attached to one of the vertices admits a closed 2-cell
embedding in the sphere, but is not 2-vertex-connected. (See 
Figure~\ref{counterexFig}.)

It is not known whether every 2-vertex-connected 
graph $\Gamma$ admits a closed 2-cell embedding.
The ``strong orientable embedding conjecture" states that 
this is the case, namely, that every
2-vertex-connected graph $\Gamma$ admits a
closed $2$-cell embedding 
in some orientable surface $S$, of face width at least two
(see \cite{MoTho}, Conjecture 5.5.16).

\begin{center}
\begin{figure}
\includegraphics[scale=0.4]{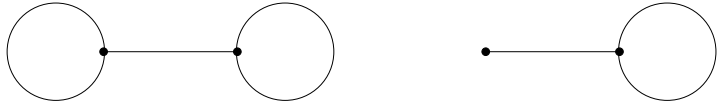}
\caption{Vertex conditions and 2-cell embeddings.
\label{counterexFig}}
\end{figure}
\end{center}

We are now ready for the promised improvement of Lemma~\ref{expli1}.

\begin{prop}\label{2faces}
Let $\Gamma$ be a graph with at least $3$ vertices and with no looping edges,
which is closed-$2$-cell embedded in an orientable surface $S$.
Then, if any two of the faces have at most one edge in common, the map $\tau$ is injective.
\end{prop}

\proof It suffices to show that, under these conditions on the graph $\Gamma$,
the second condition of Lemma~\ref{expli1} is automatically satisfied, so that
only the first condition remains to be checked. That is, we show that 
every edge of $\Gamma$ is in the boundary of two faces.

Assume an edge is not in the boundary of two faces. Then that 
edge must bound the same face on both of its sides, as in Figure \ref{edgeone}.
The closure of the face is a cell, by assumption. Let $\gamma$ be a 
path from one side of the edge to the other. Since $\gamma$ splits 
the cell into two connected components, it follows that removing
the edge splits $\Gamma$ into two connected components, hence
$\Gamma$ is not $2$-edge-connected. However, 
as recalled in Remark~\ref{closed2cell},
the fact that $\Gamma$ has at least $3$ vertices and
no looping edges and it admits a closed $2$-cell embedding implies that 
$\Gamma$ is $2$-vertex-connected, hence in particular it is 1PI by the first part of
Lemma~\ref{2edgevertex}, and this gives a contradiction.
\endproof

\begin{center}
\begin{figure}
\includegraphics[scale=.6]{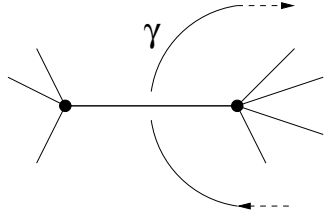}
\caption{An edge not in the boundary of two faces.
\label{edgeone}}
\end{figure}
\end{center}

The condition that $\Gamma$ has at least $3$ vertices and no looping
edges is necessary for Proposition~\ref{2faces}.  
For example, the second graph shown in Figure~\ref{counterexFig}
does not satisfy the property that each edge is in the boundary of two faces;
in the case of this graph, clearly the map $\tau$ is not injective.

\smallskip

Here is another direct consequence of the previous embedding results.

\begin{prop}\label{2conn}
Let $\Gamma$ be a $3$-edge-connected graph, with at least $3$ vertices and no looping edges,
admitting a closed-$2$-cell embedding $\iota: \Gamma \hookrightarrow S$ 
with face width $fw(\Gamma,\iota)\ge 3$. 
Then the maps $\tau_i$, $\tau$ are all injective.
\end{prop}

\proof The result of Proposition~\ref{2faces} shows that the second
condition stated in Lemma~\ref{expli1} is automatically satisfied,
so the only thing left to check is that the first condition stated in 
Lemma~\ref{expli1} holds. Assume that two faces $F_1$, $F_2$ have more 
than one edge in common, see Figure \ref{edgetwo}.
Since $F_1$, $F_2$ are (path-)connected, there are paths
$\gamma_i$ in $F_i$ connecting corresponding sides of the
edges. With suitable care, it can be arranged that 
$\gamma_1 \cup \gamma_2$ is a closed path 
$\gamma$ meeting $\Gamma$ in~$2$ points, see Figure \ref{edgetwo}.
Since the embedding has face width $\ge 3$, $\gamma$
must be null-homotopic in the surface, and in particular it
splits it into two connected components. This implies that 
$\Gamma$ is split into two connected components by 
removing the two edges, hence $\Gamma$ cannot be
$3$-edge-connected. 
\endproof

\begin{center}
\begin{figure}
\includegraphics[scale=.6]{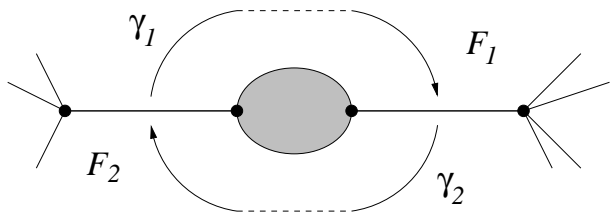}
\caption{Two faces with more than one edge in common.
\label{edgetwo}}
\end{figure}
\end{center}

The 3-edge-connectivity hypothesis in Proposition~\ref{2conn} can be
viewed as the next step strengthening of the 1PI condition, 
\cf~Remark~\ref{Landau}.
Similarly, the condition of the face width of the embedding 
$fw(\Gamma,\iota)\geq 3$ is the next step strengthening of 
the condition $fw(\Gamma,\iota)\geq 2$ conjecturally implied 
by 2-vertex-connectivity. 

In fact, if we enhance in Proposition~\ref{2conn} the 3-edge-connected
hypothesis with 3-vertex-connectivity (see Lemma~\ref{3edgever}), we can refer to
a result of graph theory (\cite{MoTho}, Proposition~5.5.12) which shows that
for a 3-vertex-connected graph it is equivalent to admit an embedding
with $fw(\Gamma,\iota)\geq 3$ and to have the {\em wheel neighborhood 
property}, that is, every vertex of $\Gamma$ has a wheel neighborhood.
Another equivalent condition to $fw(\Gamma,\iota)\geq 3$ for a 
3-vertex-connected graph is that the loops determined by the faces
of the embedding as in Lemma~\ref{2cellH1} are either disjoint or
their intersection is just a vertex or a single edge (\cite{MoTho}, Proposition~5.5.12).
For example, we can formulate an analog of Proposition~\ref{2conn} in the following way.

\begin{cor}\label{3conn}
Let $\Gamma$ be a 3-vertex-connected graph such that each vertex has a wheel 
neighborhood. Then the maps $\tau_i$ and $\tau$ of \eqref{taumap}, \eqref{mapstaui}
are all injective.
\end{cor}

The results derived in this section thus identify classes of graphs
that satisfy simple geometric properties for which the injectivity 
of the map $\tau$ holds.

\subsection{Dependence on $\Gamma$}\label{DoG}

The preceding results refer to the injectivity of the maps $\tau_i$, $\tau$
determined by a given graph $\Gamma$, where $\tau$ maps an affine space
$\A^n$ (where $n$ is the number of internal edges of $\Gamma$) to
$\A^{\ell^2}$ (where $\ell$ is the number of loops of $\Gamma$), by
means of the matrix $M_\Gamma(t)$. The whole matrix $M_\Gamma(t)$
depends of course on the graph $\Gamma$. However, the injectivity
of $\tau$ may be detected by a suitable submatrix. In the following
statement, choose a basis for $H_1(\Gamma,\Z)$ as prescribed in 
Lemma~\ref{2cellH1}; thus, $f-1=\ell-2g$ rows of $M_\Gamma(t)$ 
correspond to the `internal' faces in an embedding of $\Gamma$.

\begin{lem}\label{lm2gminor}
For a graph $\Gamma$ with at least 3 vertices and no looping edges 
that is closed-$2$-cell embedded by $\iota: \Gamma \hookrightarrow S$ 
in an orientable surface $S$ of genus $g$,
the map defined by the $(\ell -2g)\times (\ell -2g)$ minor in the matrix 
$M_\Gamma(t)$ which corresponds to the loops that are boundaries of 
faces on $S$ is injective if and only if the map $\tau$ is injective.
\end{lem}

\proof
Indeed, under the given assumptions, every edge appears in the loop
corresponding to some internal face of the embedding. The argument
proving Lemma~\ref{expli1} shows that the given minor determines
the injectivity of $\tau$.
\endproof

A further refinement of the foregoing considerations will allow us to obtain
statements that will be to some extent independent of $\Gamma$, and
only hinge on $\ell=b_1(\Gamma)$ and the genus $g$ of $\Gamma$. 

We have pointed out earlier (\S\ref{Dhagp}) that 
$\det M_\Gamma(t)$ does not depend on the choice of orientation for 
the loops of $\Gamma$. It is however advantageous to make a coherent 
choice for these orientations.
We are now assuming that we have chosen a closed $2$-cell embedding
of $\Gamma$ into an orientable surface of genus $g$; such an embedding
has $f$ faces, where $\ell=2g + f -1$; we can arrange $M_\Gamma(t)$
so that the first $f-1$ rows correspond to the $f-1$ loops determined by
the `internal' faces of the embedding.

On each face, choose the positive
orientation (counterclockwise with respect to an outgoing normal vector). 
Then each edge-variable in common between two faces $i$, $j$ 
will appear with a minus sign in the entries $(i,j)$ and $(j,i)$ of
$M_\Gamma(t)$. These entries are both in the $(\ell-2g)\times (\ell-2g)$
upper-left minor, which is the minor singled out in Lemma~\ref{lm2gminor}.

The upshot is that in the cases covered by the above results
(such as Proposition~\ref{2faces}), the edge variables $t_e$ can all be
obtained by either pulling-back entries $-x_{ij}$ with $1\le i < j\le \ell-2g$,
or a sum 
\[
x_{i1}+x_{i2}+\cdots + x_{i,\ell-2g}
\]
with $1\le i\le \ell-2g$. Note that these expressions only depend on 
$\ell$ and $g$; it follows that all components of the image 
$\tau(\partial\sigma_n)$ in $\A^{\ell^2}$ of the boundary of the simplex 
$\sigma_n$ can be realized as pull-backs
of subspaces of $\A^{\ell^2}$ from a list which {\em only depends on the 
number $\ell-2g$\/} ($=f-1$, where $f$ is the number of faces in a closed
$2$-cell embedding of $\Gamma$).
This observation essentially emancipates the domain of integration in
the integral appearing in the statement
of Lemma~\ref{periodDell} from the specific graph $\Gamma$.

We will return to this point in \S\ref{PeriodSec}, \cf~Proposition~\ref{loopgenus}.

\subsection{More general graphs}\label{Mgg}
The previous combinatorial statements were obtained under the assumption
that the graphs have no looping edges. However, the statement can then
be generalized easily to the case with looping edges using the following
observation.

\begin{lem}\label{addlooping}
Let $\Gamma$ be a graph obtained by attaching a looping edge at
a vertex of a given graph $\Gamma'$. Then the map $\tau_\Gamma$ of
\eqref{taumap} is injective if and only if $\tau_{\Gamma'}$ is.
\end{lem}

\proof Let $t$ be the variable assigned to the looping edge and $t_e$ the
variables assigned to the edges of $\Gamma'$. The matrix $M_\Gamma(t,t_e)$
is of the block diagonal form
$$ M_\Gamma(t,t_e) = \left(\begin{array}{cc} t & 0 \\ 0 & M_{\Gamma'}(t_e) \end{array}\right). $$
This proves the statement.
\endproof

This allows us to extend the results of Proposition~\ref{2conn} and Corollary~\ref{3conn} 
to all graphs obtained by attaching an arbitrary number of looping edges at the vertices of
a graph satisfying the hypothesis of Proposition~\ref{2conn} or Corollary~\ref{3conn}. 

\begin{cor}\label{withloops}
Let $\Gamma$ be a graph such that, after removing all the looping edges,
the remaining graph is 3-vertex-connected with a wheel neighborhood at each 
vertex. Then the maps $\tau_i$, $\tau$ are all injective.
\end{cor}

\medskip

We can further extend the class of graphs to which the results of this section apply
by including those graphs that are obtained from graphs satisfying the hypotheses
of  Proposition~\ref{2faces}, Proposition~\ref{2conn}, Corollary~\ref{3conn}, or
Corollary~\ref{withloops} by subdividing edges.

Let $e_n$ be the edge of $\Gamma$ that is subdivided in two edges $e_n'$ and
$e_n''$ to form the graph~$\Gamma'$. The effect on the graph polynomial is
$$ \Psi_{\Gamma'}(t_1,\ldots,t_{n-1},t_n', t_n'') =\Psi_\Gamma(t_1,\ldots, t_{n-1},
t_n'+t_n''), $$
since the spanning trees of $\Gamma'$ are obtained by adding either $e_n'$ 
or $e_n''$ to those spanning trees of $\Gamma$ that do not contain $e_n$ and by
replacing $e_n$ with $e_n' \cup e_n''$ in the spanning trees of $\Gamma$ that
contain $e_n$. Thus, notice that in this case the injectivity of the map $\tau$ is 
not preserved by the operation of splitting edges. However, one can check directly
that this operation does not affect the nature of the period computed by the Feynman
integral, as the following result shows, so that any result that will show that the
Feynman integral is a period of a mixed Tate motive for a class of graphs with no
valence two vertices will automatically extend to graphs obtained by splitting edges.

\begin{prop}\label{edgediv}
Let $\Gamma'$ be a graph obtained from a given graph $\Gamma$ by subdividing
one of the edges by inserting a valence two vertex. Then the parametric Feynman
integral for $\Gamma'$ will be of the form
\begin{equation}\label{UGammaEdge}
\int_{\sigma_n} \frac{P_\Gamma(t,p)^{-(n+1)+D\ell/2} t_n \omega_n }
{\Psi_\Gamma(t)^{-(n+1) +(\ell+1)D/2}} ,
\end{equation}
with $n=\# E_{int}(\Gamma)$.
\end{prop}

\proof When one subdivides an edge as above, the Feynman rules imply that
one finds as corresponding Feynman integral an expression of the form
$$ \int \frac{\delta(\sum_i \epsilon_{v,i} k_i + \sum_j \epsilon_{v,j} p_j)}
{q_1\cdots q_{n-1} q_n^2}  \frac{d^Dk_1}{(2\pi)^D}\cdots
 \frac{d^Dk_n}{(2\pi)^D}, $$
where the $q_i(k_i)=k_i^2 + m^2$ are the quadratic forms that give the
propagators associated to the internal edges of the graph. We have
used the constraint $\delta(k_n-k_{n+1})$ for the two momentum variables
associated to the two parts of the split edge, so that we find $q_n^2$ in
the denominator. One then uses the usual formula
$$ \frac{1}{q_1^{a_1}\cdots q_n^{a_n}} = \frac{\Gamma(a_1+\cdots+a_n)}{\Gamma(a_1)\cdots
\Gamma(a_n)} \int_{\R^n_+} \frac{t_1^{a_1-1}\cdots t_n^{a_n-1} \delta(1-\sum_i t_i)}
{(t_1 q_1+\cdots t_n q_n)^{a_1+\cdots +a_n}} $$
to obtain the parametric form of the Feynman integral. In our case this gives
$$  \frac{1}{q_1 \cdots q_{n-1} q_n^2} = n! \int_{\sigma_n} \frac{t_n\, dt_1 \cdots dt_n}
{(t_1 q_1+\cdots t_n q_n)^{n+1}}. $$
Thus, one obtains the parametric form of the Feynman integral as
$$ \int \frac{d^D x_1 \cdots d^D x_\ell}{(\sum_i t_i q_i)^{n+1}} =
C_{\ell,n+1} \det(M_\Gamma(t))^{-D/2} V_\Gamma(t,p)^{-(n+1) +D\ell/2}, $$
where $V_\Gamma(t,p)=P_\Gamma(t,p)/\Psi_\Gamma(t)$ and with
$$ C_{\ell,n+1}= \int\frac{d^D x_1 \cdots d^D x_\ell}{(1+\sum_k x_k^2)^{n+1}}. $$
This gives \eqref{UGammaEdge}.
\endproof

In particular Proposition~\ref{edgediv} shows that the parametric Feynman integral for
the graph $\Gamma'$ is still a period of the same type as that of the graph $\Gamma$,
since it is still a period associated to the complement of the graph hypersurface 
$\hat X_\Gamma$ and evaluated over the same simplex $\sigma_n$. Only the algebraic
differential form changes from $\Psi_\Gamma^{-D/2}V_\Gamma(t,p)^{-n+D\ell/2} \omega_n$
to $\Psi_\Gamma^{-D/2}V_\Gamma(t,p)^{-(n+1)+D\ell/2} t_n \omega_n$, but this does not
affect the nature of the period, at least in the ``stable range" where $D$ is sufficiently large
($D\ell/2 > n$).

\medskip

\section{The motive of the determinant hypersurface
complement}\label{MotSec} 

Our work in \S\ref{DetSec} and~\S\ref{Grcfe} relates the complexity of
a Feynman integral over a graph satisfying suitable combinatorial conditions
to the complexity of the motive 
$$ \m(\A^{\ell^2}\smallsetminus \hat\cD_\ell, \hat\Sigma_\Gamma 
\smallsetminus (\hat\Sigma_\Gamma \cap \hat\cD_\ell)) $$
whose realizations give the relative cohomology of the pair of the complement
of the determinant hypersurface and a normal crossings divisor $\hat\Sigma_\Gamma$
containing the image of the boundary $\tau_\Gamma(\partial\sigma_n)$, as in Lemma
\ref{tauinjXi} below (see Corollary~\ref{pairmot}, Proposition~\ref{2faces} and ff.).

In this section we exhibit an
explicit filtration of the complement of the determinant hypersurface, 
from which we can directly prove that the motive of 
$\A^{\ell^2} \smallsetminus \hat \cD_\Gamma$ 
is mixed Tate. We use this filtration to compute explicitly the class of 
$\A^{\ell^2} \smallsetminus \hat \cD_\Gamma$
in the Grothendieck group of varieties, as well as the class of 
the projective version $\P^{\ell^2-1}\smallsetminus \cD_\ell$.  

Notice that the mixed Tate nature of
the motive of the determinant hypersurface also follows directly 
from the results of Belkale--Brosnan \cite{BeBro},  or from those of Biglari
\cite{Biglari1}, \cite{Biglari2},  but we prefer to give here a very explicit
computation, which will be useful as a preliminary for the similar
but more involved analysis of the loci that contain the boundary 
of the domain of integration that we discuss in the following sections.

\subsection{The motive}\label{SecDellMot}

As we already argued, it is more natural to consider the graph hypersurfaces 
$\hat X_\Gamma$ in the affine space $\A^n$, instead of the projective $X_\Gamma$
in $\P^{n-1}$. Thus, here also we work with the affine space 
$\A^{\ell^2}$ parametrizing $\ell\times\ell$ matrices. The cone $\hat
\cD_\ell$ over the determinant hypersurface consists of matrices of
rank~$< \ell$. Realizing the complement of $\hat \cD_\ell$ in $\A^{\ell^2}$
amounts then to `parametrizing' matrices $M$ of rank exactly~$\ell$.

It is clear how this should be done:
\begin{itemize}
\item[---] The first row of $M$ must be a nonzero vector $v_1$;
\item[---] The second row of $M$ must be a vector $v_2$ that is nonzero
modulo $v_1$;
\item[---] The third row of $M$ must be a vector $v_3$ that is nonzero
modulo $v_1$ and $v_2$;
\item[---] And so on.
\end{itemize}

To formalize this construction, let $E$ be a fixed $\ell$-dimensional
vector space, and work inductively. The first steps of the
construction are as follows.

\begin{itemize}
\item[---] Denote by $\cW_1$ the variety $E\smallsetminus \{0\}$;
\item[---] Note that $\cW_1$ is equipped with a trivial vector bundle $E_1
=E\times \cW_1$, and with a line bundle $S_1:=L_1\subseteq E_1$ whose fiber
over $v_1\in \cW_1$ consists of the line spanned by $v_1$;
\item[---] Let $\cW_2\subseteq E_1$ be the complement $E_1\smallsetminus
L_1$;
\item[---] Note that $\cW_2$ is equipped with a trivial vector bundle $E_2
=E\times \cW_2$, and {\em two\/} line subbundles of $E_2$: the pull-back of 
$L_1$ (still denoted $L_1$) and the line-bundle $L_2$ whose fiber over
$v_2\in \cW_2$ consists of the line spanned by $v_2$;
\item[---] By construction, $L_1$ and $L_2$ span a rank-2 subbundle
$S_2$ of $E_2$;
\item[---] Let $\cW_3\subseteq E_2$ be the complement $E_2\smallsetminus
S_2$;
\item[---] And so on.
\end{itemize}

Inductively: at the $k$-th step, this procedure produces a variety $\cW_k$,
endowed with $k$ line bundles $L_1,\dots,L_k$ spanning a rank-$k$
subbundle $S_k$ of the trivial vector bundle $E_k:=E\times \cW_k$.
If $S_k\subsetneq E_k$, define $\cW_{k+1}:=E_k\smallsetminus S_k$. 
Let $E_{k+1}=E\times \cW_{k+1}$, and define line subbundles $L_1,\dots,
L_k$ to be the pull-backs of the like-named line bundles on $\cW_k$; and let 
$L_{k+1}$ be the line bundle whose fiber over $v_{k+1}$ is the line
spanned by $v_{k+1}$. The line bundles $L_1,\dots,L_{k+1}$ span
a rank-$k+1$ subbundle $S_{k+1}$ of $E_{k+1}$, and the construction
can continue. The sequence stops at the $\ell$-th step, where 
$S_\ell$ has rank~$\ell$,
equal to the rank of $E_\ell$, so that
$E_\ell \smallsetminus S_\ell = \emptyset$.

\begin{lem}\label{VkDell}
The variety $\cW_\ell$ constructed as above 
is isomorphic to $\A^{\ell^2}\smallsetminus
\hat D_\ell$.
\end{lem}

\proof Each variety $\cW_k$ maps to $\A^{\ell^2}$ as follows: a point of
$\cW_k$ determines $k$ vectors $v_1,\dots,v_k$, and can be mapped
to the matrix whose first $k$ rows are $v_1,\dots,v_k$ resp.~(and the 
remaining rows are~$0$). By construction, this matrix has rank 
exactly~$k$. Conversely, any such rank~$k$ matrix is the image of
a point of $\cW_k$, by construction.
\endproof

In particular, we have the following result on the bundles $S_k$ involved  
in the construction described above.

\begin{lem}\label{trivialSk}
The bundle $S_k$ over the variety $\cW_k$ is trivial for all $1\leq k\leq \ell$. 
\end{lem}

\proof Points of $\cW_k$ are parameterized by $k$-tuples of vectors
$v_1,\ldots,v_k$ spanning $S_k\subseteq \K^\ell \times \cW_k = E_k$. 
This means precisely that the map
\[
\K^k \times \cW_k \stackrel{\alpha}{\to} S_k 
\]
defined by
\[
\alpha: ((c_1,\ldots,c_r),(v_1,\ldots,v_r)) \mapsto c_1 v_1+\cdots + c_r v_r 
\]
is an isomorphism.
\endproof

\smallskip

Recall that, given a triangulated category $\cD$, a full subcategory 
$\cD'$ is a triangulated subcategory if and only if 
it is invariant under the shift $T$ of $\cD$ and for any distinguished
triangle 
$$ A \to B \to C \to A[1] $$
for $\cD$ where $A$ and $B$ are in $\cD'$ there is an isomorphism
$C\simeq C'$ with $C'$ also in $\cD'$. A full triangulated subcategory
$\cD'\subset \cD$ is thick if it is closed under direct sums. 

Let $\cM\cD_\K$ be the Voevodsky triangulated category of mixed
motives over a field $\K$, \cite{Voev}. The triangulated category
$\cD\cM\cT_\K$ of mixed Tate motives is the full triangulated thick
subcategory of $\cM\cD_\K$ generated by the Tate objects $\Q(n)$. It
is known that, over a number field $\K$, there is a canonical
$t$-structute on $\cD\cM\cT_\K$ and one can therefore construct 
an abelian category $\cM\cT_\K$ of mixed Tate motives (see
\cite{Levine}).

We then have the following result on the nature of the motive 
of the determinant hypersurface complement.

\begin{thm}\label{DellMTM}
The determinant hypersurface complement 
$\A^{\ell^2}\smallsetminus \hat\cD_\ell$ 
defines an object in the category $\cD\cM\cT_\K$ 
of mixed Tate motives. 
\end{thm}

\proof First recall that by Proposition~4.1.4 of \cite{Voev}, 
over a field $\K$ of characteristic zero a closed embedding $Y\subset X$
determines a distinguished triangle  
$$ \m(Y) \to \m(X) \to \m(X\smallsetminus Y) \to \m(Y)[1] $$
in $\cM\cD_\K$. Here we use the notation $\m(X)$ for the motivic
complex with compact support denoted by $\underline{C^c_*}(X)$ in
\cite{Voev}. In particular, if $\m(Y)$ and $\m(X)$ are in $\cD\cM\cT_\K$
then $\m(X\smallsetminus Y)$ is isomorphic to an object in
$\cD\cM\cT_\K$, by the property of full triangulated subcategories
recalled above. Similarly, using the invariance of $\cD\cM\cT_\K$
under the shift, if $\m(Y)$ and $\m(X\smallsetminus Y)$ are in
$\cD\cM\cT_\K$ then $\m(X)$ is isomorphic to an object in
$\cD\cM\cT_\K$. 

We also know (see \S 1.2.3 of \cite{BloMM}) that in the Voevodsky
category $\cM\cD_\K$ one inverts the morphism $X\times \A^1 \to
X$ induced by the projection, so that taking the product with an
affine space $\A^k$ is an isomorphism at the level of the
corresponding motives and for the motivic complexes with compact
support this gives $\m(X\times \A^1)=\m(X)(-1)[2]$, see Corollary~4.1.8
of \cite{Voev}. Thus, for any given $\m(X)$ in $\cD\cM\cT_\K$,
the motive $\m(X\times \A^k)$ is obtained from $\m(X)$ by Tate twists
and shifts, hence it is also in $\cD\cM\cT_\K$. 

These two properties of the derived category $\cD\cM\cT_\K$ of mixed
Tate motives suffice to show that the motive of the 
affine hypersurface complement $\A^{\ell^2}\smallsetminus
\hat\cD_\ell$ is mixed Tate,
\begin{equation}\label{DellMT}
\m(\A^{\ell^2}\smallsetminus \hat\cD_\ell) \in Obj(\cD\cM\cT_\Q).
\end{equation}
In fact, one sees from the inductive construction of
$\A^{\ell^2}\smallsetminus \hat\cD_\ell$ described above that at each
step we are dealing with varieties defines over $\K=\Q$ and we now
show that, at each step, the corresponding motives are mixed Tate.

Single points obviously belong to the category of mixed Tate motives. 
At the first step, one takes the complement $\cW_1$ of a point in an affine
space, which gives a mixed Tate motive by the first observation above
on distinguished triangles associated to closed embeddings. At the
next step one considers the complement of the line bundle $S_1$ inside
the trivial vector bundle $E_1$ over $\cW_1$. Again, both $\m(S_1)$ and
$\m(E_1)$ are mixed Tate motives, since both are products by affine spaces
by Lemma~\ref{trivialSk} above, hence $\m(E_1\smallsetminus S_1)$ is
also mixed Tate. The same argument shows that, for all $1\leq k \leq
\ell$, the motive $\m(E_k\smallsetminus S_k)$ is mixed Tate, by
repeatedly using Lemma~\ref{trivialSk} and the two properties of
$\cD\cM\cT_\Q$ recalled above.
\endproof

\subsection{The class in the Grothendieck ring}\label{GrRingSect}

Lemma~\ref{VkDell} suffices to obtain an explicit formula for the 
class in the Grothendieck ring of varieties of the complement of the
determinant hypersurface. 
This is of course well-known: see for example \cite{BeBro}, \S3.3.

\begin{prop}\label{DellGrclass}
In the affine case the class in the Grothendieck
ring of varieties is
\begin{equation}\label{classDellA}
[\A^{\ell^2}\smallsetminus \hat \cD_\ell]=  
\bL^{\binom \ell 2} \prod_{i=1}^\ell (\bL^i -1)
\end{equation}
where $\bL$ is the class of $\A^1$. In the projective case, the class is
\begin{equation}\label{classDellP}
[\P^{\ell^2-1}\smallsetminus \cD_\ell] = \bL^{\binom \ell 2}
\prod_{i=2}^\ell (\bL^i -1). 
\end{equation}
\end{prop}

\proof Using Lemma~\ref{VkDell} one sees inductively that the class
of $\cW_k$ is given by
\begin{equation}\label{classVk}
\begin{array}{rl}
[\cW_k] = & (\bL^\ell - 1)(\bL^\ell - \bL)(\bL^\ell-\bL^2)\cdots (\bL^\ell
-\bL^{k-1}) \\[3mm]
= & \bL^{\binom k2} (\bL^\ell -1) (\bL^{\ell-1}-1)\cdots
(\bL^{\ell-k+1}-1). \end{array}
\end{equation}
This completes the proof.
\endproof

The class \eqref{classDellP} can be written equivalently in the form
\begin{equation}\label{classLAT}
[\P^{\ell^2-1}\smallsetminus D_\ell]= (\bL [\P^1] \bT)\cdot
(\bL^2 [\P^2] \bT)\cdot (\bL^3[\P^3]\bT) 
\cdots (\bL^{\ell-1}[\P^{\ell-1}]\bT),
\end{equation}
where $\bL=[\A^1]$ and $\bT=[\bG_m]$ is the class of the multiplicative
group. 
Here the motive $\bL^\ell [\P^1]\cdots [\P^{\ell-1}]$ can be thought of as the motive of
the ``variety of frames".

\begin{ex}\label{class23} {\rm In the cases $\ell=2$ and $\ell=3$, the
class of $\P^{\ell^2-1}\smallsetminus \cD_\ell$ is given, respectively,
by 
$$ \bL^3 - \bL \ \ \ \text{ and } \ \ \ \bL^8-\bL^5-\bL^6+\bL^3. $$
(Note however that, for $\ell\ge 5$, coefficients other than $0$, $\pm
1$ appear in the class.) 
Thus, the class $[\cD_\ell]$ is given, for $\ell=2$ and $\ell=3$ by the
expressions 
$$ [\cD_2] = \bL^2+2\bL+1=(\bL+1)^2  $$
$$ [\cD_3]= \bL^7+2\bL^6+2\bL^5+\bL^4+\bL^2+\bL+1=(\bL^3-\bL+1)
(\bL^2+\bL+1)^2. $$
The $\ell=2$ case is otherwise evident: $\cD_2$ is the set
of rank-1, $2\times 2$--matrices, and as such it may be realized as $\P^1
\times \P^1$, with the indicated class. The $\ell=3$ case can also be easily
verified independently.} \end{ex}

\section{Relative cohomology and mixed Tate motives}\label{PeriodSec}

We now assume that $\Gamma$ is a graph satisfying the condition
studied in \S\ref{DetSec} and~\S\ref{Grcfe}: the map~$\tau$ is
injective. By Proposition~\ref{2faces}, this is the case if $\Gamma$ has
at least $3$ vertices, no looping edges, and is closed-$2$-cell embedded
in an orientable surface in such a way that any two of the faces determined
by the embedding have at most one edge in common.
Proposition~\ref{2conn} and Corollary~\ref{3conn} provide us with specific 
combinatorial conditions ensuring that this is the case. 
For instance, all 3-edge connected planar graphs are included in this class.

Also note that by the considerations in \S\ref{Mgg} (especially 
Lemma~\ref{addlooping} and Proposition~\ref{edgediv}), any estimate
for the complexity of Feynman integrals for graphs satisfying these
conditions generalizes automatically to the larger class of graphs obtained 
from those considered here by adding arbitrarily many looping edges, and 
by arbitrarily subdividing edges.

\subsection{Algebraic simplexes and normal crossing divisors}
In our setting and under the injectivity assumption, the property that the Feynman
integral \eqref{Int} is a period of a mixed Tate motive (modulo divergences) would
follow from showing that a certain relative cohomology is a realization of
a mixed Tate motive. Instead of the relative cohomology
$$ H^{n-1}(\P^{n-1}\smallsetminus X_\Gamma, \Sigma_n \smallsetminus (\Sigma_n 
\cap X_\Gamma)) $$
considered in \cite{BEK}, \cite{Blo}, we consider here a different relative
cohomology, where the hypersurface complement $\P^{n-1}\smallsetminus X_\Gamma$
is replaced by the complement $\P^{\ell^2-1}\smallsetminus \cD_\ell$ of the determinant
hypersurface, or better its affine counterpart $\A^{\ell^2}\smallsetminus \hat\cD_\ell$, and
instead of the algebraic simplex $\Sigma_n=\{ t:\, t_1\cdots t_n =0 \}$, we consider a
locus $\hat\Sigma_\Gamma$ 
in $\A^{\ell^2}$ that pulls back to the algebraic simplex $\Sigma_n$ 
under the map $\tau$ of \eqref{taumap} and that consists of a union of $n$ linear subspaces of codimension one in $\A^{\ell^2}$ that meet the image of $\A^n$ 
under $\tau$ along divisors with normal crossings. 
The following observation is a direct consequence of the construction of
the matrix $M_\Gamma(t)$ (\cf~\S\ref{Dhagp}).

\begin{lem}\label{tauinjXi}
Suppose given a graph $\Gamma$ such that the corresponding maps $\tau$ 
and $\tau_i$ are injective. Then the $n$ coordinates $t_i$
associated to the internal edges of $\Gamma$ can be written as preimages 
via the (injective) map $\tau: \A^n \to \A^{\ell^2}$ of $n$ linear
subspaces $X_i$ of codimension~1 in $\A^{\ell^2}$. 
These $n$ subspaces form a divisor $\hat\Sigma_\Gamma$ with normal
crossings in $\A^{\ell^2}$.
\end{lem}

\proof Consider the various possible cases for a specific edge listed in 
Lemma~\ref{mapstaui}.
In the third case listed there, where there are two loops $\ell_i$, $\ell_j$ containing 
$e$, and not having
any other edge in common, the variable $t_e$ is immediately expressed as the 
pullback to $\A^n$ of a coordinate in $\A^{\ell^2}$. 
Consider then the second case listed in Lemma~\ref{mapstaui}, where 
an edge $e$ belongs to a single loop $\ell_i$. Under the assumption that
the map $\tau_i$ is injective, then any linear combination of the variables corresponding
to the edges in the $i$-th loop may be written as a linear combination
of coordinates of the $i$-th row. 
\endproof

The considerations in \S\ref{DoG} allow us to improve this observation,
by passing to a larger normal
crossing divisor, so that one can generate all the $\hat\Sigma_\Gamma$ 
from the components of a single normal crossings divisor $\hat\Sigma_{\ell,g}$ 
that only depends on the number of loops of the graph and on the minimal genus 
of the embedding of the graph on a Riemann surface. We formalize this remark
as follows.

\begin{prop}\label{loopgenus} 
There exists a normal crossings divisor $\hat\Sigma_{\ell,g}\subset \A^{\ell^2}$,
which is a union of $N= \binom{f}{2}$ linear spaces
\begin{equation}\label{hatSigmaXi}
\hat \Sigma_{\ell,g}:=X_1\cup \cdots\cup X_N ,
\end{equation}
such that, for all graphs  $\Gamma$ with $\ell$ loops
and genus $g$ closed 2-cell embedding,  the preimage under $\tau=\tau_\Gamma$
of the union $\hat\Sigma_\Gamma$ of a subset  of components of $\hat\Sigma_{\ell,g}$
is the algebraic simplex $\Sigma_n$ in $\A^n$.  More explicitly, the divisor $\hat\Sigma_{\ell,g}$ can
be described by the $N = \binom{f}{2}$ equations
\begin{equation}\label{eqhatSigma}
\left\{
\begin{array}{rl}
x_{ij} &=0\quad 1\le i<j\le f-1 \\
x_{i1}+\cdots + x_{i,f-1} &=0 \quad 1\le i\le f-1 ,
\end{array}
\right.
\end{equation}
where $f=\ell -2g +1$ is the number of faces of the embedding.
\end{prop}

\proof Using Lemma~\ref{lm2gminor}, we know that the injectivity of an 
$(\ell -2g)\times (\ell -2g)$ minor of the matrix $M_\Gamma$ suffices to control 
the injectivity of the map $\tau$. 
We can in fact arrange so that the minor is the upper-left part of the
$\ell\times \ell$ ambient matrix. 
Then, as in Lemma~\ref{tauinjXi}, the hyperplanes in $\A^n$ associated to the 
coordinates $t_i$ can be obtained by
pulling back linear spaces along  this minor.
On the diagonal of the $(f-1)\times (f-1)$
submatrix we find all edges making up each face, with a positive
sign. It follows that the pull-backs of the equations \eqref{eqhatSigma}
produce a list of all the edge variables, possibly with redundancies.
The components of $\hat\Sigma_{\ell,g}$ that form the divisor $\hat\Sigma_\Gamma$
are selected by eliminating those components of $\hat\Sigma_{\ell,g}$ that contain 
the image of the graph hypersurface  (\ie
coming from the zero entries of the matrix $M_\Gamma(t)$).
\endproof

\medskip

Thus, for every $\Gamma$ satisfying the conditions recalled at the beginning
of the section (for example, every 3-edge connected planar graph, or every
graph obtained from one of these by adding looping edges or subdividing
edges), {\em the nature of period
appearing as a Feynman integral over $\Gamma$ in the sense explained
in \S\ref{DetSec} is controlled by the motive 
\begin{equation}\label{relmotGamma}
 \m (\A^{\ell^2}\smallsetminus \hat\cD_\ell, \hat\Sigma_\Gamma \smallsetminus
(\hat\cD_\ell \cap \hat\Sigma_\Gamma)), 
\end{equation}
for a normal crossing divisor
$\hat\Sigma_\Gamma\subset \A^{\ell^2}$ consisting of a subset of components of 
the fixed\/} (for given $\ell$ and $g$) {\em normal crossing divisor 
$\hat\Sigma_{\ell,g}\subset \A^{\ell^2}$ introduced above.\/}

More explicitly,  the boundary of the
topological simplex $\sigma_n$, that is, the domain of integration of the
Feynman integral in Lemma~\ref{periodDell}, satisfies 
\begin{equation}\label{boundaryhatSigma}
\tau(\partial \sigma_n)\subset \hat \Sigma_\Gamma \subset \hat\Sigma_{\ell,g}. 
\end{equation}
Thus, the main goal here
will be to understand the motivic nature of the complement
\begin{equation}\label{hatSigmaMot}
\hat\Sigma_\Gamma \smallsetminus (\hat\cD_\ell \cap \hat\Sigma_\Gamma).
\end{equation}

Since $\hat \Sigma_\Gamma$ consists of components from the fixed
normal crossing divisor $\hat\Sigma_{\ell,g}$, this question will be recast
in terms that only depend on $\ell$ and $g$: 
we show in Corollary~\ref{Xicap} below that, using the inclusion--exclusion
principle applied to the components of $\hat\Sigma_{\ell,g}$, it is possible to
answer these questions simultaneously for all the divisors $\hat\Sigma_\Gamma$,
for all graphs with $\ell$ loops and genus $g$, by investigating the nature of a
motive constructed out of the intersections of the components of the divisor
$\hat\Sigma_{\ell,g}$. 

Notice in fact that one can derive the case of $\hat\Sigma_{\ell,g}$ from
the case of $g=0$, since $\hat\Sigma_{\ell,g} \subseteq \hat\Sigma_{\ell,0}$,
corresponding to an $(\ell -2g)\times (\ell -2g)$ minor of the matrix 
$M_\Gamma(t)$.

\smallskip

There are general and explicit conditions (see \cite{Gon}, Proposition~3.6)
implying that the relative cohomology of a pair $(X,Y)$ comes from a mixed Tate 
motive $\m(X,Y)$ (see also \cite{GonMan} for a concrete
application to the geometric case of moduli spaces of curves). In general,
these rely on assumptions on the divisors involved and
their associated stratification, which may not directly apply to the cases 
considered here. We discuss here a direct approach to constructing
stratifications of our loci $\hat\Sigma_{\ell,g}\smallsetminus (\hat\cD_\ell
\cap \hat\Sigma_{\ell,g})$ that can be used to investigate the nature of the 
motive \eqref{relmotGamma}.

\subsection{Inclusion--exclusion}\label{IE}
The procedure we follow will be the one outlined above, based on the
divisors $\hat\Sigma_{\ell,g}$ and the inclusion--exclusion principle.
Since we already know by the results of \S \ref{MotSec}  that the complement
$X=\A^{\ell^2}\smallsetminus \hat\cD_\ell$ is a mixed Tate motive, we aim at providing 
a direct argument showing that 
$Y=\hat\Sigma_\Gamma \smallsetminus (\hat\Sigma_\Gamma \cap \hat\cD_\ell)$ 
also is a mixed Tate motive. The same argument used in~\S\ref{MotSec} based 
on the distinguished triangles in the Voevodsky triangulated category of mixed 
Tate motives \cite{Voev} would then show that the relative cohomology of the 
pair $(X,Y)$ comes from an object $\m(X,Y) \in Obj(\cD\cM\cT_\Q)$.

As a first step we transform the problem of a complement in a union of linear 
spaces into an equivalent formulation in terms of intersections of linear spaces, 
using inclusion--exclusion.
For a collection $\{ Z_i \}_{i\in I}$ of varieties $Z_i$ we set 
\begin{equation}\label{ZIcirc}
Z_I^\circ :=(\cap_{i\in I} Z_i)\smallsetminus 
(\cup_{j\not\in I} Z_j).
\end{equation}
Notice that, for all $I$,
\[
\cap_{i\in I} Z_i = \amalg_{J\supseteq I} Z_J^\circ .
\]
This is a {\em disjoint\/} union. We then have the following result.

\begin{lem}\label{incexc}
Let $Z_1,\dots,Z_m$ be varieties; assume that the intersections
$\cap_{i\in I} Z_i$ are mixed Tate, for all nonempty $I\subseteq \{1,\dots,m\}$.
Then $Z_1\cup \cdots\cup Z_m$ is mixed Tate.
\end{lem}

\proof We want to show that $Z_I^\circ$ is mixed Tate for all nonempty $I\subseteq \{1,\dots,m\}$.
To see this, notice that it is true by hypothesis for $I=\{1,\dots,m\}$,
since in this case $Z_I^\circ =\cap_{i\in I} Z_i$. Thus, it suffices to
prove that if it is true for all $I$ with $|I|>k$, then it is true for all 
$I$ with $|I|=k$ (provided $k\ge 1$). Recall that, as we already used in \S  \ref{MotSec} above,
the distinguished triangles in the Voevodsky category of mixed Tate motives imply that, if
$X\hookrightarrow Y$ is a closed embedding, and 
$U=Y \smallsetminus X$ the complement, then if any two of $X, Y, U$ are 
mixed Tate so is the third as well. The result then 
follows from the combined use of this property, the hypothesis, and 
the identity
\[
Z_I^\circ= (\cap_{i\in I} Z_i) \smallsetminus
(\amalg_{J\supsetneq I} Z_J^\circ)\quad.
\]
Since we have
\[
Z_1\cup \cdots \cup Z_m=\amalg_{I\ne \emptyset} Z_I^\circ\quad,
\]
we conclude that the union $Z_1\cup \cdots \cup Z_m$ is mixed Tate, 
again by the property of mixed Tate motives mentioned above.
\endproof

Now, we have observed that for every graph $\Gamma$ with $\ell$ loops
and genus $g$ (and satisfying the condition specified at the beginning of
the section) the divisor $\hat\Sigma_\Gamma$ consists of components
of the divisor $\hat\Sigma_{\ell,g}$. Therefore, the strata of 
$\hat\Sigma_\Gamma$ are unions of strata from $\hat\Sigma_{\ell,g}$.
We can then reformulate our main problem as follows.

\begin{cor}\label{Xicap}
Let, as above, $\hat\Sigma_{\ell,g} =X_1\cup \cdots \cup X_N$ and let $\hat\Sigma_\Gamma$
be the divisors constructed out of subsets of components of $\hat\Sigma_{\ell,g}$, associated
to the individual graphs. Then, for all graphs $\Gamma$ with $\ell$ loops and genus $g$,
the complement $\hat\Sigma_\Gamma \smallsetminus (\hat \cD_\ell \cap \hat\Sigma_\Gamma)$ 
is mixed Tate if the locus
\begin{equation}\label{XicapDell}
(\cap_{i\in I} X_i) \smallsetminus \hat \cD_\ell
\end{equation}
is mixed Tate for all $I\subseteq \{1,\dots,N\}$, $I\ne \emptyset$.
\end{cor}

\proof This is a direct consequence of Lemma~\ref{incexc}. \endproof

Corollary~\ref{Xicap} encapsulates the main reformulation of our problem,
mentioned at the end of \S\ref{intro}: the target becomes that of
proving that the loci $(\cap_{i\in I} X_i) \smallsetminus \hat \cD_\ell$
determined by the normal crossing divisor $\hat\Sigma_{\ell,g}$ are
mixed Tate.
This result shows that, although in principle one is working with a different
divisor $\hat\Sigma_\Gamma$ for each graph $\Gamma$, in fact it suffices to
consider the divisor $\hat\Sigma_{\ell,g}$, for fixed number of loops $\ell$
and genus $g$.
It is conceivable that the loci associated to a specific graph (that is, to
a specific choice of components of $\hat\Sigma_{\ell,g}$) may be 
mixed Tate while the loci corresponding to the whole divisor 
$\hat\Sigma_{\ell,g}$ is not. As we are seeking an explanation that
would imply that {\em all\/} periods arising from Feynman integrals
are periods of mixed Tate motives, we will optimistically venture 
that {\em all loci $(\cap_{i\in I} X_i) \smallsetminus \hat \cD_\ell$ may in 
fact turn out to be mixed Tate,\/} for all $\ell$ and for $g=0$: 
by Corollary~\ref{Xicap}, it would follow that all complements
$\hat\Sigma_\Gamma \smallsetminus (\hat \cD_\ell \cap 
\hat\Sigma_\Gamma)$ are mixed Tate, for {\em all\/} graphs $\Gamma$
(satisfying our running combinatorial hypothesis).

Our task is now to formulate this working hypothesis as a more
concrete problem. The intersection $\cap_{i\in I} X_i$
is a linear subspace of codimension $|I|$ in $\A^{\ell^2}$;
in general, the intersection of a linear subspace with the determinant
is {\em not\/} mixed Tate (for example, the intersection of a general
$\A^3$ with $\hat D_3$ is a cone over a genus-$1$ curve).
Thus, we have to understand in what sense the intersections
$\cap_{i\in I} X_i$ appearing in Corollary~\ref{Xicap} are special;
the following lemma determines some key features of these subspaces.

\begin{lem}\label{linFi}
Let $E$ be a fixed $\ell$-dimensional vector space, as in \S \ref{SecDellMot} above.
Every $I\subseteq \{1,\dots,N\}$ as above determines
a choice of linear subspaces $V_1,\dots,V_\ell$ of $E$, such that
\begin{equation}\label{capXiFk}
\cap_{k\in I} X_k = \{(v_1,\dots,v_\ell)\in \A^{\ell^2} \,|\, \forall i, v_i\in V_i\} .
\end{equation}
(Here, we denote an $\ell\times\ell$ matrix in $\A^{\ell^2}$ by its 
$\ell$ row-vectors $v_i\in E$.)

Further, $\dim V_i \ge i-1$.
Further still, there exists a basis $(e_1,\dots,e_\ell)$ of $E$ such that each 
space $V_i$ is the span of a subset (of cardinality $\ge i-1$) of the vectors 
$e_j$.
\end{lem}

\proof 
Recall (Proposition~\ref{loopgenus}) that the components $X_k$ of 
$\hat\Sigma_{\ell,g}$ consist of matrices for which either
the $(i,j)$ entry $x_{ij}$ equals~$0$, for $1\le i<j\le \ell-2g$, or 
\[
x_{i1}+\cdots +x_{i,\ell-2g}=0
\]
for $1\le i\le \ell-2g$. Thus,
each $X_k$ consists of $\ell$-tuples $(v_1,\dots,v_\ell)$ for which
exactly one row $v_i$ belongs to a fixed hyperplane of $E$,
and more precisely to one of the hyperplanes
\begin{equation}\label{eqhypers}
x_1+\cdots+x_{\ell-2g}=0\quad,\quad
x_2=0\quad,\quad
\cdots\quad,\quad
x_{\ell-2g}=0
\end{equation}
(with evident notation).
The statement follows by choosing $V_i$ to be the intersection of the
hyperplanes corresponding to the $X_k$ in row~$i$, among those listed 
in \eqref{eqhypers}. Since there are at most $\ell-2g-i+1$ hyperplanes $X_k$
in the $i$-th row, 
\[
\dim V_i \ge \ell-(\ell-2g-i+1)=2g+i-1\ge i-1\quad.
\]
Finally, to obtain the basis $(e_1,\dots, e_\ell)$ mentioned in the statement,
simply choose the basis dual 
to the basis $(x_1+\cdots +x_{\ell-2g}, x_2,\dots, x_\ell)$ of the dual space 
to $E$.
\endproof

\subsection{The main questions}\label{Tmq}

In view of Lemma~\ref{linFi}, for any choice $V_1,\dots,V_\ell$ of subspaces
of an $\ell$-dimensional space~$E$, let
\begin{equation}\label{FkDell}
\F(V_1,\ldots, V_\ell):= \{(v_1,\dots,v_\ell)\in \A^{\ell^2} \,|\, \forall k, v_k\in V_k\} \smallsetminus \hat \cD_\ell
\end{equation}
denote the complement of the determinant hypersurface in the set of
matrices determined by $V_1,\dots,V_\ell$. An optimistic version of the
question we are led to is:\medskip

{\bf Question~$\text{I}_\ell$.} {\em Let $V_1,\dots,V_\ell$ be subspaces of an
$\ell$-dimensional vector space. Is the locus
$\F(V_1,\dots,V_\ell)$ mixed Tate?\/}\medskip

By Corollary~\ref{Xicap} and Lemma~\ref{linFi}, an affirmative answer to
Question~$\text{I}_\ell$ implies that the complement $\hat\Sigma_\Gamma
\smallsetminus (\hat D_\ell \cap \hat \Sigma_\Gamma)$ is mixed Tate
for all graphs $\Gamma$ with $\ell$ loops and satisfying the combinatorial 
condition given at
the beginning of this section. Modulo divergence issues,
this would imply that all Feynman integrals corresponding to these graphs
are periods of mixed Tate motives. We will give an affirmative answer to
Question~$\text{I}_\ell$ for $\ell\le 3$, in \S\ref{MotFrameSec}.

As Lemma~\ref{linFi} is in fact more precise, the same conclusion would
be reached by answering affirmatively the following weak version of 
Question~$\text{I}_\ell$:\medskip

{\bf Question~$\text{II}_\ell$.} {\em Let $(e_1,\dots,e_\ell)$ be a basis of an 
$\ell$-dimensional vector space. For $i=1,\dots,\ell$, let $V_i$ be a 
subspace spanned by a choice of $\ge i-1$ basis vectors. 
Is $\F(V_1,\dots,V_\ell)$ mixed Tate?\/}\medskip

Notice that, when $V_k =E$ for all $k$, both questions reproduce the statement 
about the hypersurface complement $\A^{\ell^2}\smallsetminus \hat\cD_\ell$ 
proved in \S \ref{SecDellMot}.
One might expect that a similar inductive procedure would provide a simple
approach to these questions. It is natural to consider the following apparent
refinement of Question~$\text{I}_\ell$ for $1\le r\le \ell$ (and we could similarly 
consider an analogous refinement Question~$\text{II}'_{\ell,r}$ of 
Question~$\text{II}_\ell$):\medskip

{\bf Question~$\text{I}'_{\ell,r}$.} {\em In a vector space $E$ of dimension~$\ell$, 
and for any choice of subspaces $V_1,\dots, V_r$ of $E$, let
\[
\F_\ell(V_1,\dots,V_r)=\text{$\{(v_1,\dots,v_r) \,|\, v_i\in V_i$ and $\dim 
\langle v_1,\dots v_r\rangle=r\}$}\quad.
\]
Is the locus $\F_\ell(V_1,\dots,V_r)$ mixed Tate?\/}\medskip

Question~$\text{I}_\ell$ is then the same as Question~$\text{I}'_{\ell,\ell}$;
and Question~$\text{I}'_{\ell,r}$ is obtained by taking $V_{r+1}=\cdots
=V_\ell=E$ in Question~$\text{I}_\ell$: thus, answering 
Question~$\text{I}_\ell$ is equivalent to answering 
Question~$\text{I}'_{\ell,r}$ for all $r\le \ell$.

Now, for all $\ell$, the case $r=1$ is immediate: $\F_\ell(V_1)$ consists
of all nonzero vectors in $V_1$, which is trivially mixed Tate. 
One could then hope that an inductive procedure may yield a method
for increasing~$r$. This is carried out in \S\ref{MotFrameSec} for
$r=2$ and $r=3$ (in particular, we give an affirmative answer to
Question~$\text{I}_\ell$ for $\ell\le 3$); but this approach quickly leads to 
the analysis 
of several different cases, with an increase in complexity that makes
further progress along these lines seem unlikely. The main problem is 
that once all tuples $(v_1,\dots,v_k)$ of linearly independent
vectors such that $v_i\in V_i$ have been constructed, controlling
\[
\dim (V_{k+1}\cap \langle v_1,\dots,v_k\rangle)
\]
requires consideration of a range of possibilities that depend on 
the position of the vectors $v_i$ and their spans vis-a-vis the position
of the next space $V_{k+1}$. The number of these possibilities increases
rapidly.
A similar approach to the simpler (but sufficient for our purposes) Question~II 
does not appear to circumvent this problem.

\medskip

There are special cases where an inductive argument works
nicely. We mention two here.

\begin{itemize}
\item Suppose that all the $V_k$ in \eqref{FkDell} are hyperplanes in $E$.
Then  $\F(V_1,\dots,V_\ell)$ is mixed Tate.
\end{itemize}

In this case, following the inductive argument 
mentioned above, the only possibilities for $V_{k+1}\cap \langle
v_1,\dots,v_k\rangle$ are $\langle v_1,\dots,v_k\rangle$, and a 
hyperplane in $\langle v_1,\dots,v_k\rangle$. The first occurs when
\[
\langle v_1,\dots,v_k\rangle\subseteq V_k .
\]
This locus is under control, since it amounts to doing the whole
construction in $V_k$ rather than $E$, \ie one can argue by induction on the 
dimension of $E$. Thus, this locus is mixed Tate. The other case gives a locus
that is the complement of this mixed Tate variety in another mixed Tate variety, 
hence, by the same argument about closed embeddings and distinguished triangles
used in \S  \ref{MotSec}, it is also mixed Tate.

\begin{itemize}
\item Suppose $V_1\subseteq V_2\subseteq\cdots\subseteq V_r$; then
$\F_\ell(V_1,\dots,V_r)$ is mixed Tate.
\end{itemize}

Indeed, in this case $\langle v_1,\dots,v_k\rangle \subseteq V_{k+1}$
for all $k$. The condition on $v_{k+1}$ is simply $v_{k+1}\in V_{k+1}
\smallsetminus \langle v_1,\dots,v_k\rangle$, and these conditions
clearly produce a mixed Tate locus. Arguing as in \S\ref{SecDellMot},
the class of $\F_\ell(V_1,\dots,V_r)$ is immediately seen to equal
\[
(\L^{d_1}-1)(\L^{d_2}-\L)(\L^{d_3}-\L^2)\cdots (\L^{d_r}- \L^{r-1})
\]
in this case, where $d_k=\dim V_k$.

\subsection{A reformulation}

For given subspaces $V_i\subset E$, 
the inductive approach suggested by Question~$\text{I}'_{\ell,r}$ 
aims at constructing the
set of $\ell$-uples $(v_1, \ldots, v_\ell)$ with the two properties 
\begin{enumerate}
\item $v_i\in V_i$;
\item $\dim \langle v_1, \ldots, v_r \rangle = r$, for all $r$,
\end{enumerate}
and proving inductively that these loci are mixed Tate, in order to show
that the loci \eqref{FkDell} are mixed Tate. 
By (2), the sets
\[
0 \subset \langle v_1\rangle \subset \langle v_1,v_2\rangle \subset 
\cdots \subset \langle v_1,\dots,v_\ell \rangle = E
\]
form a complete flag in $E$; let $E_r=\langle v_1,\dots,v_r\rangle$.
Our main question can then be phrased in terms of these moving
complete flags:\medskip

{\bf Question~$\text{III}_\ell$.} {\em Let $V_1,\dots,V_\ell$ be subspaces
of an $\ell$-dimensional vector space $E$, and let $d_i,e_i$ be integers. 
Is the locus $\text{Flag}_{\ell, \{d_i,e_i\}}(\{V_i\})$ of complete flags
\[
0 \subset E_1 \subset E_2 \subset \cdots \subset E_\ell = E
\]
such that 
\begin{itemize}
\item $\dim E_i \cap V_i = d_i$
\item $\dim E_i \cap V_{i+1} = e_i$
\end{itemize} 
mixed Tate?\/}\medskip

An affirmative answer to this question (for all choices of $d_i$, $e_i$) would
give an affirmative answer to our main Question~$\text{I}_\ell$. Indeed, 
the locus $\F(V_1,\dots,V_\ell)$ is a fibration on the
locus $\text{Flag}_{\ell, \{d_i,e_i\}}(\{V_i\})$ determined in Question~$\text{III}_\ell$. 
Concretely, the procedure constructing the tuples $(v_1,\dots,v_\ell)$ in 
$\F(V_1,\dots,V_\ell)$ over a flag $E_\bullet$ in this locus 
is:
\begin{itemize}
\item Choose $v_1\in (E_1\cap V_1)\smallsetminus \{ 0 \}$;
\item Choose $v_2 \in (E_2\cap V_2)\smallsetminus  (E_1 \cap V_2)$;
\item Choose $v_3 \in (E_3 \cap V_3)\smallsetminus (E_2 \cap V_3)$;
\item etc.
\end{itemize}

The class of $\F(V_1,\dots,V_\ell)$ in the Grothendieck group would then
be computed as a sum of terms
\[
[\text{Flag}_{\ell, \{d_i,e_i\}}(\{V_i\})] \cdot (\L^{d_1}-1)(\L^{d_2}-\L^{e_1})
(\L^{d_3}-\L^{e_2})\cdots (\L^{d_r}-\L^{e_{r-1}})\quad.
\]
The set of flags $E_\bullet$ satisfying conditions analogous to those 
specified in Question~$\text{III}_\ell$ with respect to all terms of a
fixed flat $E'_\bullet$ (that is: with prescribed $\dim(E_i\cap E'_j)$
for all $i$ and~$j$) is a cell of the corresponding Schubert variety in
the flag manifold.

It follows that $\text{Flag}_{\ell, \{d_i,e_i\}}(\{V_i\})$ is a disjoint union of cells, 
and thus certainly mixed Tate, if the $V_i$'s form a complete flag. This gives
a high-brow alternative viewpoint for the last case mentioned in~\S\ref{Tmq}. 

By the same token, the set of flags $E_\bullet$ for which $\dim E_i\cap F$
is a fixed constant is a union of Schubert cells in the flag manifold, for all 
subspaces $F$.
It follows that the locus $\text{Flag}_{\ell, \{d_i,e_i\}}(\{V_i\})$ of 
Question~$\text{III}_\ell$ is an {\em intersection of unions of Schubert
cells\/} in the flag manifold. Such loci were studied e.g.~in \cite{GeSe},
\cite{GeGoMaSe}, \cite{ShaShaVai1}, \cite{ShaShaVai2}.

\section{Motives and manifolds of frames}\label{MotFrameSec}

The manifolds of $r$-frames in a given vector space are defined as follows.

\begin{defn}\label{framesDef}
Let $\F(V_1, \ldots, V_r)\subset V_1 \times \cdots \times V_r$ 
denote the locus of $r$-tuples of linearly independent
vectors in a vector space, where each $v_i$ is constrained to belong to the given
subspace $V_i$. 
\end{defn}

These are the loci appearing in Question~$\text{I}'_{\ell,r}$; we now omit the
explicit mention of the dimension $\ell$ of the ambient space.
The question we consider here is the one formulated
in \S\ref{Tmq},
namely to establish when the motive of the manifold of frames
$\F(V_1,\ldots,V_r)$ is mixed Tate. 
A possible strategy to answering this question is based on the following
simple observations. 

\begin{lem}\label{fromproj}
Let $V_1,\dots,V_r$ be subspaces of a given vector space $V$.
Let $v_r\in V_r$, and let $\pi: V \to V':=V/\langle v_r\rangle$ be the natural
projection. Let $v_1,\dots,v_{r-1}$ be vectors such that $v_i\in V_i$, and
$\pi(v_1),\dots, \pi(v_{r-1})$ are linearly independent. Then
$v_1,\dots, v_r$ are linearly independent.
\end{lem}

\proof
The dimension of $\pi(\langle v_1,\dots,v_{r-1}\rangle) = \langle \pi(v_1),\dots,
\pi(v_{r-1})\rangle$ is $r-1$ by hypothesis, therefore $\dim \pi^{-1}(\pi
(\langle v_1,\dots,v_{r-1}\rangle))=r$. Since $\pi^{-1}(\pi
(\langle v_1,\dots,v_{r-1}\rangle))\subseteq \langle v_1,\dots,v_r\rangle$,
it follows that $\dim \langle v_1,\dots,v_r\rangle=r$, as needed.
\endproof

A second equally elementary remark is that for a given $v'\ne 0$ in the quotient
$V/\langle v_r\rangle$, and letting as above $\pi$ denote the projection
$V \to V/\langle v_r\rangle$, $\pi^{-1}(v')\cap V_i$ consists of either a single 
vector, if $v_r\not\in V_i$, or a copy of the field $k$, if $v_r\in V_i$. 

This implies the following. 

\begin{lem}\label{Salpha}
Suppose given a stratification $\{S_\alpha\}$ of $V_r$ 
with the properties that
\begin{itemize}
\item $\{S_\alpha\}$ is finer than the stratification induced on $V_r$ by the 
subspace
arrangement $V_1\cap V_r,\dots, V_{r-1}\cap V_r$, hence the number
$s_\alpha$ of spaces $V_i$ ($1\le i<r$) containing a vector $v_r\in S_\alpha$
is independent of the vector and only depends on $\alpha$.
\item For $v_r\in S_\alpha$, the class $\F_\alpha:=[\F(\pi(V_1),
\dots,\pi(V_{r-1}))]$ also depends only on $\alpha$, and not on the chosen 
vector $v_r\in S_\alpha$.
\end{itemize} 
Then the class in the Grothendieck group satisfies
\begin{equation}\label{FSalpha}
[\F (V_1,\dots,V_r)]=\sum_\alpha \L^{s_\alpha} \cdot 
[\F_\alpha]\cdot [S_\alpha]\quad.
\end{equation}
\end{lem}

\proof
Indeed, by Lemma~\ref{fromproj} every frame in the quotient will determine
frames in $V$, and by the observation following the Lemma, there is a
whole $k^{s_\alpha}$ of frames over a given one in the quotient.
\endproof

In an inductive argument, the loci $[\F_\alpha]$ could be assumed to
be mixed Tate, and \eqref{FSalpha} would provide a strong indication that
$[\F(V_1,\dots,V_r)]$ is then mixed Tate as well. We focus here on giving
statements at the level of classes in the Grothendieck ring, for simplicity,
though these same arguments, based on constructing explicit stratifications,
can be also used to derive conclusion on the motives at the level
of the derived category of mixed motives in a way similar to what we
did in the case of the complement of the determinant hypersurface 
in \S\ref{MotSec} above.

The main question is then reduced to finding conditions under which 
a stratification of the type described here exists. We see explicitly
how the argument goes in the simplest cases of two and three subspaces.
As we discuss below, the case of three subspaces is already more involved
and exhibits some of the features one is bound to encounter, with a more
complicated combinatorics, in the more general cases.

\subsection{The case of two subspaces}\label{case2subs}

Let $V_1$, $V_2$ be subspaces
of a vector space $V$. We want to parametrize all pairs of vectors
\[
(v_1,v_2)
\]
such that $v_1\in V_1$, $v_2\in V_2$, and $\dim\langle v_1,v_2\rangle=2$.
This locus can be decomposed into two pieces (which may be empty), 
defined by the following prescriptions:
\begin{enumerate}
\item Choose $v_1\in V_1\smallsetminus (V_1\cap V_2)$, and $v_2\in 
V_2\smallsetminus \{0\}$;
\item Choose $v_1\in (V_1\cap V_2)\smallsetminus \{0\}$, and $v_2\in V_2
\smallsetminus \langle v_1\rangle$.
\end{enumerate}
It is clear that each of these two recipes produces linearly independent
vectors, and that (1) and (2) exhaust the ways in which this can be done.
So $\F(V_1,V_2)$ is the union of the corresponding loci.
Pairs $(v_1,v_2)$ as in (1) range over the locus 
$(V_1\smallsetminus (V_1\cap V_2))\times (V_2\smallsetminus \{0\})$, which is
clearly mixed Tate. As for (2), realize it as follows: 
\begin{itemize}
\item Consider the projective space $\P(V_1\cap V_2)$, and the 
trivial bundles $\cV_{12}\subseteq \cV_2$ with fiber 
$V_1\cap V_2 \subseteq V_2$.
\item $\cV_{12}$ contains the tautological line bundle $\cO_{12}(-1)$
over $\P(V_1\cap V_2)$, hence this line bundle is naturally 
contained in $\cV_2$ as well.
\item Then the pairs $(v_1,v_2)$ as in (2) are obtained by choosing
a point $p\in \P(V_1\cap V_2)$, a vector $v_1\ne 0$ in the fiber of
$\cO_{12}(-1)$ over $p$, and a vector $v_2$ in the fiber of 
$\cV_2\smallsetminus \cO_{12}(-1)$ over $p$.
\end{itemize}
It is clear that this description also produces a mixed Tate motive.

Note that the prescriptions given as (1) and (2) suffice to compute the
class in the Grothendieck group.

\begin{lem}\label{r=2}
The class in the Grothendieck group of the manifold of frames $\F(V_1,V_2)$ is of the form
\begin{equation}\label{r2frames}
[\F(V_1,V_2)]=\L^{d_1+d_2}-\L^{d_1}-\L^{d_2}-\L^{d_{12}+1}
+\L^{d_{12}}+\L\quad,
\end{equation}
where $d_i = \dim V_i$ and $d_{12}=\dim (V_1\cap V_2)$.
\end{lem}

\proof The two loci (1) and (2) respectively have classes 
\begin{enumerate}
\item $(\L^{d_1}-\L^{d_{12}})(\L^{d_2}-1)$;
\item $(\L^{d_{12}}-1)(\L^{d_2}-\L)$.
\end{enumerate}
The class of $\F(V_1,V_2)$ is then given by the sum
\begin{align*}
[\F(V_1,V_2)]&=(\L^{d_1+d_2}-\L^{d_1}-\L^{d_2+d_{12}}+\L^{d_{12}})
+(\L^{d_2+d_{12}}-\L^{d_{12}+1}-\L^{d_2}+\L) \\
&=\L^{d_1+d_2}-\L^{d_1}-\L^{d_2}-\L^{d_{12}+1}
+\L^{d_{12}}+\L .
\end{align*}
This gives \eqref{r2frames}.
\endproof

Notice that the expression for $[\F(V_1,V_2)]$ is symmetric in $V_1$ and $V_2$, though
the two individual contributions (1) and (2) are not. Of course a more symmetric description
of the locus can be obtained by subdividing it into four cases according to whether
$v_1$, $v_2$ are or are not in $V_1\cap V_2$.

\subsection{The case of three subspaces}

We are given three subspaces $V_1$, $V_2$, $V_3$
of a vector space, and we want to parametrize all triples of linearly
independent vectors $(v_1,v_2,v_3)$ with $v_i\in V_i$. As above,
$d_i$ will stand for the dimension of $V_i$, and $d_{ij}$ for 
$\dim (V_i\cap V_j)$. Further, 
let $d_{123}=\dim (V_1\cap V_2
\cap V_3)$, and $D=\dim (V_1+V_2+V_3)$.

Notice that now the information on the dimension $D$ is also needed 
and does not follow from the other data. This can be seen easily by
thinking of the cases of three distinct lines
spanning a 3-dimensional vector space or of three 
distinct coplanar lines. These configurations only differ in the
number $D$, yet the set of linearly independent triples is nonempty
in the first case, empty in the second.  

We proceed as follows. Given a choice of $v_3\in V_3$, consider the projection
$\pi: V \to V':=V/\langle v_3\rangle$; in $V'$ we have the images
$\pi(V_1), \pi(V_2)$, to which we can apply the case $r=2$ analyzed
above. As we have seen, 
$\F (V_1',V_2')$ is determined by the dimensions of $V'_1$, $V'_2$,
and $V'_1\cap V'_2$. Thus, we need a stratification of $V_3$ such
that, for $v_3\in V_3$ and denoting as above by $\pi$ the projection
$V\to V/\langle v_3\rangle$, the dimensions of the spaces
\[
\pi(V_1)\quad,\quad \pi(V_2)\quad,\quad \pi(V_1)\cap \pi(V_2)
\]
are constant along strata.

\begin{lem}\label{r3strata}
The following 5 loci give a stratification of $V_3\smallsetminus \{0\}$ with the
properties of Lemma~\ref{Salpha}.
\begin{enumerate}
\item $S_{123}:=(V_1\cap V_2\cap V_3)\smallsetminus \{0\}$;
\item $S_{13}:=(V_1\cap V_3)\smallsetminus (V_1\cap V_2\cap V_3)$;
\item $S_{23}:=(V_2\cap V_3)\smallsetminus (V_1\cap V_2\cap V_3)$;
\item $S_{(12)3}:=((V_1+V_2)\cap V_3)\smallsetminus ((V_1\cup V_2)\cap V_3)$;
\item $S_3:=V_3\smallsetminus ((V_1+V_2)\cap V_3)$.
\end{enumerate}
\end{lem}

\proof First observe that
$$ \dim \pi(V_i)=\left\{
\aligned
d_i \qquad &\text{if $v_3\not\in V_i$} \\
d_i-1 \quad &\text{if $v_3\in V_i$} \\
\endaligned
\right. $$
As for $\dim(\pi(V_1)\cap \pi(V_2))$, note that
\[
\dim(\pi(V_1)\cap \pi(V_2))
=\dim(\pi(V_1))+\dim(\pi(V_2))-\dim(\pi(V_1) + \pi(V_2))
\]
and
\[
\dim(\pi(V_1) + \pi(V_2))
=\dim(\pi(V_1+V_2))=
\left\{
\aligned
\dim(V_1+V_2) \qquad &\text{if $v_3\not\in V_1+V_2$} \\
\dim(V_1+V_2)-1 \quad &\text{if $v_3\in V_1+V_2$} \\
\endaligned
\right.\]
It follows easily that the three numbers $\dim \pi(V_1)$, $\dim \pi(V_2)$, 
$\dim(\pi(V_1\cap V_2))$ are constant along the strata. More explicitly
one has the following data.

\begin{center}
\begin{tabular}{|c||c|c|c|}
\hline
 & $\dim \pi(V_1)$ & $\dim \pi(V_2)$ & $\dim (\pi(V_1)\cap \pi(V_2))$ \\
\hline
\hline
$S_{123}$ & $d_1-1$ & $d_2-1$ & $d_{12}-1$\\
\hline
$S_{13}$ & $d_1-1$ & $d_2$ & $d_{12}$\\
\hline
$S_{23}$ & $d_1$ & $d_2-1$ & $d_{12}$\\
\hline
$S_{(12)3}$ & $d_1$ & $d_2$ & $d_{12}+1$\\
\hline
$S_{3}$ & $d_1$ & $d_2$ & $d_{12}$\\
\hline
\end{tabular}
\end{center}

For example, in the fourth (and most interesting) case, $\dim \pi(V_1)=d_1$
and $\dim \pi(V_2)=d_2$ since $v_3\not\in V_i$ if $v_3\in S_{(12)3}$;
$\dim \pi(V_1+V_2)=\dim (V_1+V_2)-1$ since $v_3\in V_1+V_2$; and
hence
\begin{align*}
\dim \pi(V_1)\cap \pi(V_2) &=\dim V_1 + \dim V_2 - \dim(V_1+V_2) + 1\\
&= \dim (V_1\cap V_2) +1 = d_{12}+1\quad.
\end{align*}

Lemma~\ref{r=2} converts this information into the list of the classes
$[\F_\alpha]$ and one obtains the following list of cases.

\begin{center}
\begin{tabular}{|c||c|}
\hline
 & $[\F_\alpha]$ \\
\hline
\hline
$S_{123}$ & $\L^{d_1+d_2-2}-\L^{d_1-1}-\L^{d_2-1}-\L^{d_{12}}
+\L^{d_{12}-1}+\L$\\
\hline
$S_{13}$ & $\L^{d_1+d_2-1}-\L^{d_1-1}-\L^{d_2}-\L^{d_{12}+1}
+\L^{d_{12}}+\L$\\
\hline
$S_{23}$ & $\L^{d_1+d_2-1}-\L^{d_1}-\L^{d_2-1}-\L^{d_{12}+1}
+\L^{d_{12}}+\L$\\
\hline
$S_{(12)3}$ & $\L^{d_1+d_2}-\L^{d_1}-\L^{d_2}-\L^{d_{12}+2}
+\L^{d_{12}+1}+\L$\\
\hline
$S_{3}$ & $\L^{d_1+d_2}-\L^{d_1}-\L^{d_2}-\L^{d_{12}+1}
+\L^{d_{12}}+\L$ \\
\hline
\end{tabular}
\end{center}

The number $s_\alpha$ is immediately read off the geometry. The last
ingredient consists of the class $[S_\alpha]$, which is also essentially
immediate. The only item that deserves attention is the dimension of
$(V_1+V_2)\cap V_3$. This is
\[
\dim (V_1+V_2)+\dim V_3-\dim (V_1+V_2+V_3) =
\dim (V_1+V_2) + d_3 - D\quad;
\]
and as
\[
\dim (V_1+V_2)=\dim V_1+\dim V_2-\dim(V_1\cap V_2) = d_1+d_2-d_{12}
\quad,
\]
we have
\[
\dim((V_1+V_2)\cap V_3)=d_1+d_2+d_3-D-d_{12}\quad.
\]
With this understood one obtains the following list of cases.

\begin{center}
\begin{tabular}{|c||c|c|}
\hline
 & $[S_\alpha]$ & $s_\alpha$ \\
\hline
\hline
$S_{123}$ & $\L^{d_{123}}-1$ & $2$\\
\hline
$S_{13}$ & $\L^{d_{13}}-\L^{d_{123}}$ & $1$\\
\hline
$S_{23}$ & $\L^{d_{23}}-\L^{d_{123}}$ & $1$\\
\hline
$S_{(12)3}$ & $\L^{d_1+d_2+d_3-D-d_{12}}-\L^{d_{13}} -\L^{d_{23}} 
+\L^{d_{123}} $ & $0$\\
\hline
$S_{3}$ & $\L^{d_3}-\L^{d_1+d_2+d_3-D-d_{12}}$ & $0$ \\
\hline
\end{tabular}
\end{center}

This completes the proof. \endproof

We can now apply equation \eqref{FSalpha}, and this gives the following result.

\begin{lem}\label{r=3}
The class of $\F(V_1,V_2,V_3)$ in the Grothendieck group is of the form 
\begin{equation}\label{GrF3}
\begin{array}{c}
[\F(V_1,V_2,V_3)]
=(\L^{d_1}-1)(\L^{d_2}-1)(\L^{d_3}-1) \\[2mm]
-(\L-1)\left((\L^{d_1}-\L)(\L^{d_{23}}-1)+(\L^{d_2}-\L)(\L^{d_{13}}-1)
+(\L^{d_3}-\L)(\L^{d_{12}}-1)\right) \\[2mm]
+(\L-1)^2\left(\L^{d_1+d_2+d_3-D}-\L^{d_{123}+1}\right)
+(\L-1)^3\quad.
\end{array}
\end{equation}
\end{lem}

Notice once again that the expression \eqref{GrF3} is symmetric in
$V_1$, $V_2$, $V_3$, unlike the contributions of the individual strata.
Slightly more refined considerations, in the style of those sketched
in \S\ref{case2subs}, prove that $[\F(V_1,V_2,V_3)]$ is in fact mixed Tate.

\medskip

In principle, the procedure applied here should work for a larger number
of subspaces: the main task amounts to the determination of a stratification
of the last subspace satisfying the properties given in Lemma~\ref{Salpha}.
This is bound to be rather challenging for $r\ge 4$: already for $r=4$ one 
can produce examples for which the closures of the strata are not linear 
subspaces. This is in fact the case already for $V_1,\dots,V_3$ planes
in general position in a $4$-dimensional ambient space $E$: the 
unique quadric cone containing $V_1$, $V_2$, $V_3$ is the closure
of a stratum in a stratification of $V_4=E$ satisfying the properties
listed in Lemma~\ref{Salpha}.

\subsection{Graphs with three loops}

One can apply the formula of Lemma~\ref{r=3} to compute explicitly
the motive (as a class in the Grothendieck group) for the locus
\begin{equation}\label{hatSigmaMot3}
\hat\Sigma_{3,0} \smallsetminus (\hat\Sigma_{3,0}\cap \hat\cD_3)
\end{equation}
of intersection of the divisor with normal crossings $\hat\Sigma_{\ell,g}$ 
of \eqref{hatSigmaXi} with the complement of the determinant hypersurface,
in the case of (planar) graphs with three loops. 

As pointed out in the discussion following Corollary~\ref{Xicap}, 
studying $\hat\Sigma_{3,0}$ suffices in order to get analogous 
information for $\hat\Sigma_\Gamma$ for every graph with three 
loops and satisfying the condition specified at the beginning of
\S\ref{PeriodSec} (guaranteeing that the corresponding
map $\tau$ is injective). The divisor $\hat\Sigma_{3,0}$ is the divisor 
corresponding to the ``wheel with three spokes''
graph (the skeleton of the tetrahedron).

This graph has matrix $M_\Gamma(t)$ given by
 \[
\begin{pmatrix}
t_1+t_2+t_5 & -t_1 & -t_2 \\
-t_1 & t_1+t_3+t_4 & -t_3 \\
-t_2 & -t_3 & t_2+t_3+t_6
\end{pmatrix}
\]
Here, $t_1,\dots,t_6$ are variables associated with the six edges of the
graph, labeled as in Figure~\ref{wheelFig}.

\begin{center}
\begin{figure}
\includegraphics[scale=.4]{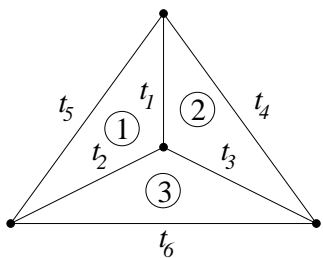}
\caption{The wheel with three spokes graph \label{wheelFig}}
\end{figure}
\end{center}

Choose the internal faces with counterclockwise orientation as
the basis of loops. Then any orientation for the edges leads 
to the matrix displayed above.
Labeling entries of the matrix as $x_{ij}$, we can obtain $t_1,\dots,t_6$
as pull-backs of the following:
\[
\left\{
\aligned 
t_1 &=-x_{12} \\
t_2 &=-x_{13} \\
t_3 &=-x_{23} \\
t_4 &=x_{21}+x_{22}+x_{23} \\
t_5 &=x_{11}+x_{12}+x_{13} \\
t_6 &=x_{31}+x_{32}+x_{33} \\
\endaligned
\right.
\]

Thus, we are considering the divisor $\hat\Sigma_{3,0}$ with
normal crossings given by the equation
\[
x_{12} x_{13} x_{23} 
(x_{11}+x_{12}+x_{13} )(x_{21}+x_{22}+x_{23})(x_{31}+x_{32}+x_{33})
=0\quad.
\]

We want to obtain an explicit description, as a class in the
Grothendieck group, of the intersection of
this locus with the complement of determinant
hypersurface $\hat\cD_3$ in $\A^9$. 
By inclusion--exclusion  (\cf~\S\ref{IE})
 this can be done by carrying out the computation
for all intersections of subsets of the components of this divisor.
Since there are $6$ components, there are $2^6=64$ such intersections.

Each of these possibilities determines a triple of subspaces $V_1,V_2,V_3$
inside the ambient~$\A^9$ (\cf~Lemma~\ref{linFi}),
corresponding to linearly independent 
vectors $v_1,v_2,v_3$, \ie the rows of the matrix $x_{ij}$, 
parameterizing points in the complement of the determinant.

Thus, to begin with, one computes for each of these cases 
the corresponding class $[\F(V_1,V_2,V_3)]$ using Lemma~\ref{r=3}.

Note that each of these classes is necessarily a multiple of $(\L-1)^3$:
indeed, once the directions of $v_1,v_2,v_3$ are specified, the set
of vectors with those directions forms a $(\C^*)^3$. We list the
classes here, divided by this constant factor $(\L-1)^3$. Each class
is marked according to the components of $\hat\Sigma_{3,0}$ containing
the corresponding locus: for example, $\bullet\bullet\circ$~$\circ\circ\bullet$
corresponds to the complement of $\hat\cD_3$ in the intersection of 
$X_1\cap X_2\cap X_6$, where $X_i$ pulls back to $t_i$ via $\tau$ as 
above (thus, $X_1\cap X_2\cap X_6$ has equations $x_{12}=x_{13}
=x_{31}+x_{32}+x_{33}=0$).

{\tiny
\begin{center}
\begin{tabular}{|c|c||c|c||c|c||c|c||}
\hline
$\bullet\bullet\bullet$~$\bullet\bullet\bullet$ & $0$ &
$\bullet\bullet\circ$~$\circ\bullet\bullet$ & $0$ &
$\bullet\circ\circ$~$\circ\bullet\bullet$ & $\bL^3$ &
$\circ\bullet\circ$~$\circ\circ\bullet$ & $\bL^3(\bL+2)$ \\
\hline
$\circ\bullet\bullet$~$\bullet\bullet\bullet$ & $0$ &
$\bullet\bullet\circ$~$\bullet\circ\bullet$ & $\bL(\bL+1)$ &
$\bullet\circ\circ$~$\bullet\circ\bullet$ & $\bL^2(\bL+1)$ &
$\circ\bullet\circ$~$\circ\bullet\circ$ & $\bL^3(\bL+1)$ \\
\hline
$\bullet\circ\bullet$~$\bullet\bullet\bullet$ & $0$ &
$\bullet\bullet\circ$~$\bullet\bullet\circ$ & $0$ &
$\bullet\circ\circ$~$\bullet\bullet\circ$ & $\bL^3$ &
$\circ\bullet\circ$~$\bullet\circ\circ$ & $\bL^3(\bL+2)$ \\
\hline
$\bullet\bullet\circ$~$\bullet\bullet\bullet$ & $0$ &
$\bullet\bullet\bullet$~$\circ\circ\bullet$ & $\bL^2$ &
$\bullet\circ\bullet$~$\circ\circ\bullet$ & $\bL(\bL^2+2\bL-1)$ &
$\circ\bullet\bullet$~$\circ\circ\circ$ & $\bL^3(\bL+1)$ \\
\hline
$\bullet\bullet\bullet$~$\circ\bullet\bullet$ & $0$ &
$\bullet\bullet\bullet$~$\circ\bullet\circ$ & $0$ &
$\bullet\circ\bullet$~$\circ\bullet\circ$ & $\bL^2(\bL+1)$ &
$\bullet\circ\circ$~$\circ\circ\bullet$ & $\bL^3(\bL+2)$ \\
\hline
$\bullet\bullet\bullet$~$\bullet\circ\bullet$ & $\bL$ &
$\bullet\bullet\bullet$~$\bullet\circ\circ$ & $\bL^2$ &
$\bullet\circ\bullet$~$\bullet\circ\circ$ & $\bL^2(\bL+1)$ &
$\bullet\circ\circ$~$\circ\bullet\circ$ & $\bL^3(\bL+1)$ \\
\hline
$\bullet\bullet\bullet$~$\bullet\bullet\circ$ & $0$ &
$\circ\circ\circ$~$\bullet\bullet\bullet$ & $0$ &
$\bullet\bullet\circ$~$\circ\circ\bullet$ & $\bL^2(\bL+1)$ &
$\bullet\circ\circ$~$\bullet\circ\circ$ & $\bL^3(\bL+2)$ \\
\hline
$\circ\circ\bullet$~$\bullet\bullet\bullet$ & $0$ &
$\circ\circ\bullet$~$\circ\bullet\bullet$ & $\bL^2(\bL+1)$ &
$\bullet\bullet\circ$~$\circ\bullet\circ$ & $0$ &
$\bullet\circ\bullet$~$\circ\circ\circ$ & $\bL^3(\bL+2)$ \\
\hline
$\circ\bullet\circ$~$\bullet\bullet\bullet$ & $0$ &
$\circ\circ\bullet$~$\bullet\circ\bullet$ & $\bL^3$ &
$\bullet\bullet\circ$~$\bullet\circ\circ$ & $\bL^2(\bL+1)$ &
$\bullet\bullet\circ$~$\circ\circ\circ$ & $\bL^3(\bL+1)$ \\
\hline
$\circ\bullet\bullet$~$\circ\bullet\bullet$ & $\bL^2$ &
$\circ\circ\bullet$~$\bullet\bullet\circ$ & $\bL^3$ &
$\bullet\bullet\bullet$~$\circ\circ\circ$ & $\bL^3$ &
$\circ\circ\circ$~$\circ\circ\bullet$ & $\bL^3(\bL+1)^2$ \\
\hline
$\circ\bullet\bullet$~$\bullet\circ\bullet$ & $\bL^2$ &
$\circ\bullet\circ$~$\circ\bullet\bullet$ & $\bL^3$ &
$\circ\circ\circ$~$\circ\bullet\bullet$ & $\bL^3(\bL+1)$ &
$\circ\circ\circ$~$\circ\bullet\circ$ & $\bL^3(\bL+1)^2$ \\
\hline
$\circ\bullet\bullet$~$\bullet\bullet\circ$ & $0$ &
$\circ\bullet\circ$~$\bullet\circ\bullet$ & $\bL^2(\bL+1)$ &
$\circ\circ\circ$~$\bullet\circ\bullet$ & $\bL^3(\bL+1)$ &
$\circ\circ\circ$~$\bullet\circ\circ$ & $\bL^3(\bL+1)^2$ \\
\hline
$\bullet\circ\circ$~$\bullet\bullet\bullet$ & $0$ &
$\circ\bullet\circ$~$\bullet\bullet\circ$ & $\bL^3$ &
$\circ\circ\circ$~$\bullet\bullet\circ$ &  $\bL^3(\bL+1)$ &
$\circ\circ\bullet$~$\circ\circ\circ$ & $\bL^3(\bL+1)^2$ \\
\hline
$\bullet\circ\bullet$~$\circ\bullet\bullet$ & $\bL^2$ &
$\circ\bullet\bullet$~$\circ\circ\bullet$ & $\bL^2(\bL+1)$ &
$\circ\circ\bullet$~$\circ\circ\bullet$ & $\bL^3(\bL+2)$ &
$\circ\bullet\circ$~$\circ\circ\circ$ & $\bL^3(\bL+1)^2$ \\
\hline
$\bullet\circ\bullet$~$\bullet\circ\bullet$ & $\bL^2$ &
$\circ\bullet\bullet$~$\circ\bullet\circ$ & $\bL^3$ &
$\circ\circ\bullet$~$\circ\bullet\circ$ & $\bL^3(\bL+2)$ &
$\bullet\circ\circ$~$\circ\circ\circ$ & $\bL^3(\bL+1)^2$ \\
\hline
$\bullet\circ\bullet$~$\bullet\bullet\circ$ & $\bL^2$ &
$\circ\bullet\bullet$~$\bullet\circ\circ$ & $\bL^3$ &
$\circ\circ\bullet$~$\bullet\circ\circ$ & $\bL^3(\bL+1)$ &
$\circ\circ\circ$~$\circ\circ\circ$ & $\bL^3(\bL+1)(\bL^2+\bL+1)$ \\
\hline
\end{tabular}
\end{center}
}

Next, one applies inclusion--exclusion to go from the class $[\F(V_1,V_2,V_3)]$
as above,which correspond to the complement of the determinant in subspaces
obtained as intersections of the 6 divisors, to classes corresponding to the
complement of the determinant in {\em the complement of smaller subspaces
in a given subspace.\/} This produces the following list of classes in
the Grothendieck group; in this table, the classes do include the common 
factor~$(\L-1)^3$.

{\tiny
\begin{center}
\begin{tabular}{|c|c||c|c||c|c||c|c||}
\hline
$\bullet\bullet\bullet$~$\bullet\bullet\bullet$ & $0$ &
$\bullet\bullet\circ$~$\circ\bullet\bullet$ & $0$ &
$\bullet\circ\circ$~$\circ\bullet\bullet$ & $\bL^2(\bL-1)^4$ &
$\circ\bullet\circ$~$\circ\circ\bullet$ & $\bL^2(\bL-1)^5$ \\
\hline
$\circ\bullet\bullet$~$\bullet\bullet\bullet$ & $0$ &
$\bullet\bullet\circ$~$\bullet\circ\bullet$ & $\bL^2(\bL-1)^3$ &
$\bullet\circ\circ$~$\bullet\circ\bullet$ & $\bL^2(\bL-1)^4$ &
$\circ\bullet\circ$~$\circ\bullet\circ$ & $\bL^2(\bL-1)^5$ \\
\hline
$\bullet\circ\bullet$~$\bullet\bullet\bullet$ & $0$ &    
$\bullet\bullet\circ$~$\bullet\bullet\circ$ & $0$ &
$\bullet\circ\circ$~$\bullet\bullet\circ$ & $\bL^2(\bL-1)^4$ &
$\circ\bullet\circ$~$\bullet\circ\circ$ & $\bL^2(\bL-1)^5$ \\
\hline
$\bullet\bullet\circ$~$\bullet\bullet\bullet$ & $0$ &
$\bullet\bullet\bullet$~$\circ\circ\bullet$ & $\bL(\bL-1)^4$ &
$\bullet\circ\bullet$~$\circ\circ\bullet$ & $\bL^2(\bL-1)^4$ &
$\circ\bullet\bullet$~$\circ\circ\circ$ & $\bL(\bL-1)^6$ \\
\hline
$\bullet\bullet\bullet$~$\circ\bullet\bullet$ & $0$ &
$\bullet\bullet\bullet$~$\circ\bullet\circ$ & $0$ &
$\bullet\circ\bullet$~$\circ\bullet\circ$ & $\bL^2(\bL-1)^4$ &
$\bullet\circ\circ$~$\circ\circ\bullet$ & $\bL(\bL^2-\bL-1)(\bL-1)^4$ \\
\hline
$\bullet\bullet\bullet$~$\bullet\circ\bullet$ & $\bL(\bL-1)^3$ &
$\bullet\bullet\bullet$~$\bullet\circ\circ$ & $\bL(\bL-1)^4$ &
$\bullet\circ\bullet$~$\bullet\circ\circ$ & $\bL(\bL-1)^5$ &
$\bullet\circ\circ$~$\circ\bullet\circ$ & $\bL^2(\bL-1)^5$ \\
\hline
$\bullet\bullet\bullet$~$\bullet\bullet\circ$ & $0$ &
$\circ\circ\circ$~$\bullet\bullet\bullet$ & $0$ &
$\bullet\bullet\circ$~$\circ\circ\bullet$ & $\bL^2(\bL-1)^4$ &
$\bullet\circ\circ$~$\bullet\circ\circ$ & $\bL^2(\bL-1)^5$ \\
\hline
$\circ\circ\bullet$~$\bullet\bullet\bullet$ & $0$ &
$\circ\circ\bullet$~$\circ\bullet\bullet$ & $\bL^2(\bL-1)^4$ &
$\bullet\bullet\circ$~$\circ\bullet\circ$ & $0$ &
$\bullet\circ\bullet$~$\circ\circ\circ$ & $\bL^2(\bL-1)^5$ \\
\hline
$\circ\bullet\circ$~$\bullet\bullet\bullet$ & $0$ & 
$\circ\circ\bullet$~$\bullet\circ\bullet$ & $\bL(\bL-1)^5$ &
$\bullet\bullet\circ$~$\bullet\circ\circ$ & $\bL^2(\bL-1)^4$ &
$\bullet\bullet\circ$~$\circ\circ\circ$ & $\bL^2(\bL-1)^5$ \\
\hline
$\circ\bullet\bullet$~$\circ\bullet\bullet$ & $\bL^2(\bL-1)^3$ &
$\circ\circ\bullet$~$\bullet\bullet\circ$ & $\bL^2(\bL-1)^4$ &
$\bullet\bullet\bullet$~$\circ\circ\circ$ & $\bL(\bL-1)^5$ &
$\circ\circ\circ$~$\circ\circ\bullet$ & $\bL(\bL^2-\bL-1)(\bL-1)^5$ \\
\hline
$\circ\bullet\bullet$~$\bullet\circ\bullet$ & $\bL(\bL-1)^4$ &
$\circ\bullet\circ$~$\circ\bullet\bullet$ & $\bL^2(\bL-1)^4$ &
$\circ\circ\circ$~$\circ\bullet\bullet$ & $\bL^2(\bL-1)^5$ &
$\circ\circ\circ$~$\circ\bullet\circ$ & $\bL^2(\bL-1)^6$ \\
\hline
$\circ\bullet\bullet$~$\bullet\bullet\circ$ & $0$ &
$\circ\bullet\circ$~$\bullet\circ\bullet$ & $\bL^2(\bL-1)^4$ &
$\circ\circ\circ$~$\bullet\circ\bullet$ & $\bL^2(\bL-1)^5$ &
$\circ\circ\circ$~$\bullet\circ\circ$ & $\bL^2(\bL-1)^6$ \\
\hline
$\bullet\circ\circ$~$\bullet\bullet\bullet$ & $0$ &
$\circ\bullet\circ$~$\bullet\bullet\circ$ & $\bL^3(\bL-1)^3$ &
$\circ\circ\circ$~$\bullet\bullet\circ$ &  $\bL^2(\bL-1)^5$ &
$\circ\circ\bullet$~$\circ\circ\circ$ & $\bL^2(\bL-1)^6$ \\
\hline
$\bullet\circ\bullet$~$\circ\bullet\bullet$ & $\bL^2(\bL-1)^3$ &
$\circ\bullet\bullet$~$\circ\circ\bullet$ & $\bL(\bL-1)^5$ &
$\circ\circ\bullet$~$\circ\circ\bullet$ & $\bL^2(\bL-1)^5$ &
$\circ\bullet\circ$~$\circ\circ\circ$ & $\bL^2(\bL-1)^6$ \\
\hline
$\bullet\circ\bullet$~$\bullet\circ\bullet$ & $\bL(\bL-1)^4$ &
$\circ\bullet\bullet$~$\circ\bullet\circ$ & $\bL^2(\bL-1)^4$ &
$\circ\circ\bullet$~$\circ\bullet\circ$ & $\bL^2(\bL-1)^5$ &
$\bullet\circ\circ$~$\circ\circ\circ$ & $\bL(\bL^2-\bL-1)(\bL-1)^5$ \\
\hline
$\bullet\circ\bullet$~$\bullet\bullet\circ$ & $\bL^2(\bL-1)^3$ &
$\circ\bullet\bullet$~$\bullet\circ\circ$ & $\bL(\bL-1)^5$ &
$\circ\circ\bullet$~$\bullet\circ\circ$ & $\bL(\bL-1)^6$ &
$\circ\circ\circ$~$\circ\circ\circ$ & $\bL(\bL^2-\bL-1)(\bL-1)^6$ \\
\hline
\end{tabular}
\end{center}
}

These are the classes of the individual strata of the stratification of
$\A^9\smallsetminus \hat\cD_3$ determined by $\hat\Sigma_{3,0}$
(including several empty strata). The sum of the classes in
this table is the class $[\A^9\smallsetminus \hat\cD_3]$, that is
$\bL^3 (\bL+1) (\bL^2+\bL+1) (\bL-1)^3$
($= (\L-1)(\bL^8-\bL^5-\bL^6+\bL^3)$, cf.~Example~\ref{class23}).

It is interesting to notice that the expressions simplify 
when one takes inclusion--exclusion into account. The cancellations due to
inclusion-exclusion mostly lead to classes of the form $\L^a (\L-1)^b$.

\smallskip

In terms of Feynman integrals, in the case of the wheel with three spokes,
we are interested in the relative cohomology
$$ H^*(\A^9\smallsetminus \hat\cD_3, \hat\Sigma_{3,0}\smallsetminus 
(\hat\cD_3\cap \hat\Sigma_{3,0})). $$
The hypersurface complement $\A^9\smallsetminus \hat\cD_3$
has class 
\begin{equation}\label{classA9D3}
[\A^9\smallsetminus \hat\cD_3]
=\bL^3 (\bL+1) (\bL^2+\bL+1) (\bL-1)^3,
\end{equation}
while the class of $\hat\Sigma_{3,0}\smallsetminus 
(\hat\cD_3\cap \hat\Sigma_{3,0})$ may be obtained as the
sum of all the classes listed above except 
$\circ\circ\circ$~$\circ\circ\circ$, which
corresponds to the choice of subspaces where 
$V_1=V_2=V_3$ is the whole space
(these are all the strata
of $\hat\Sigma_{3,0}\smallsetminus (\hat\cD_3\cap \hat\Sigma_{3,0})$) 
or, equivalently, the difference of
\eqref{classA9D3} and the last item
$\circ\circ\circ$~$\circ\circ\circ$.
This gives
\begin{multline*}
[\hat\Sigma_{3,0}\smallsetminus (\hat\cD_3\cap \hat\Sigma_{3,0})]=
\bL^3 (\bL+1) (\bL^2+\bL+1) (\bL-1)^3
-\bL(\bL^2-\bL-1)(\bL-1)^6\\
=\bL(6\bL^4-3\bL^3+2\bL^2+2\bL-1)(\bL-1)^3
\end{multline*}
The main information is carried by the class
$\circ\circ\circ$~$\circ\circ\circ$,
\begin{equation}\label{oooooo}
\bL(\bL^2-\bL-1)(\bL-1)^6.
\end{equation}

\begin{center}
\begin{figure}
\includegraphics[scale=.4]{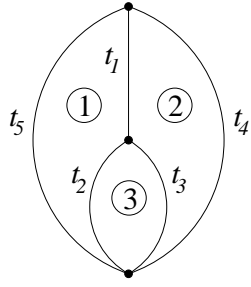}
\caption{Another graph with 3 loops and injective $\tau$ \label{3loopsFig}}
\end{figure}
\end{center}

In the case of other 3-loop graphs $\Gamma$, such as the one illustrated
in Figure~\ref{3loopsFig}, the divisor $\hat\Sigma_\Gamma$ is a union
of components of $\hat\Sigma_{3,0}$ (\cf~Proposition~\ref{loopgenus}).
The class of the locus $\hat\Sigma_\Gamma\smallsetminus 
(\hat\cD_3\cap \hat\Sigma_{3,0})$ may be obtained by adding up 
all contributions listed above, for the strata contained in 
$\hat\Sigma_\Gamma$. For the example given in Figure~\ref{3loopsFig},
these are the strata contained in the divisors $X_1,\dots,X_5$; the
corresponding classes are those marked by $***$~$***$, where at least
one of the first five $*$ is $\bullet$; or, equivalently,  the difference of
\eqref{classA9D3} and the classes marked $\circ\circ\circ$~$\circ\circ\bullet$
and $\circ\circ\circ$~$\circ\circ\circ$. The sum of these two classes is
\[
\bL(\bL^2-\bL-1)(\bL-1)^5+\bL(\bL^2-\bL-1)(\bL-1)^6
=\bL^2(\bL^2-\bL-1)(\bL-1)^5
\]
(\cf~\eqref{oooooo}), and hence
\begin{multline*}
[\hat\Sigma_\Gamma\smallsetminus (\hat\cD_3\cap \hat\Sigma_{3,0})]
=\bL^3 (\bL+1) (\bL^2+\bL+1) (\bL-1)^3
-\bL^2(\bL^2-\bL-1)(\bL-1)^5 \\
=\bL^2(5\bL^3+1)(\bL-1)^3\quad.
\end{multline*}

\section{Divergences and renormalization}\label{Divs}

Our analysis in the previous sections of this paper concentrated on the task
of showing that a certain relative cohomology is a realization of a mixed
Tate motive $\m(X,Y)$, where the loci $X$ and $Y$ are constructed, respectively,
as the complement of the determinant hypersurface and the intersection with this
complement of a normal crossing divisor that contains the image of the boundary 
of the domain of integration $\sigma_n$ under the  map $\tau_\Gamma$, for any 
graph $\Gamma$ with fixed number of loops and fixed genus. Knowing that
$\m(X,Y)$ is a mixed Tate motive implies that, when convergent, the parametric
Feynman integral for all such graphs is a period of a mixed Tate motive. This,
however, does not take into account the presence of divergences in the
Feynman integrals.

There are several different approaches to regularize and renormalize 
the divergent integrals. We outline here some of the possibilities and 
comment on how they can be made compatible with our approach.

\subsection{Blowups}

One possible approach to dealing with divergences coming from the
intersections of the divisor $\Sigma_n$ with the graph hypersurface
$X_\Gamma$ is the one proposed by Bloch--Esnault--Kreimer in \cite{BEK}, 
namely one can proceed to perform a series of blowups of strata of this 
intersection until one has separated the domain of integration from the 
hypersurface and in this way regularized the integral.

In our setting, a similar approach should be reformulated in the
ambient $\A^{\ell^2}$ and in terms of the intersection of the
determinant hypersurface $\hat\cD_\ell$ with the divisor 
$\hat\Sigma_{\ell,g}$. 
If the main question posed in \S\ref{Tmq} has an affirmative answer,
then this intersection admits a stratification by mixed Tate 
nonsingular loci. It seems likely that a suitable sequence of blow-ups would
then have the effect of regularizing the integral, while at the same time maintaining
the motivic nature of the relevant loci unaltered. We intend to
return to a more detailed analysis of this approach in future work.

\subsection{Dimensional regularization and L-functions}

Belkale and Brosnan showed in \cite{BeBro2} that dimensionally
regularized Feynman integrals can be written in the form of a
local Igusa L-function, where the coefficients of the Laurent
series expansion are periods, provided the integrals describing
them are convergent. Such periods have an explicit description
in terms of integrals on simplices $\sigma_n$ and cubes $[0,1]^r$
of algebraic differental forms 
$$ f(t)^{s_0} \omega_n \wedge \frac{f(t)-1}{(f(t)-1)t_1 +1} dt_1 \wedge \cdots\wedge
\frac{f(t)-1}{(f(t)-1)t_r +1} dt_r , $$
for $f(t)= \Psi_\Gamma(t)$ the graph polynomial. The nature of such integrals 
as periods would still be controlled by the same motivic loci that are involved in
the original parametric Feynman integral before dimensional regularization.
The result of \cite{BeBro2}  is formulated only for the case of log-divergent integrals
where only the graph polynomial $\Psi_\Gamma(t)$ is present in
the Feynman parametric form and not the polynomial $P_\Gamma(t,p)$.
The result was extended to the more general non-log-divergent case 
by Bogner and Weinzierl in \cite{BoWe1}.

In this approach, if there are singularities in the integrals that
compute the coeffients of the Laurent series expansion of
the local Igusa L-function giving the dimensionally regularized
Feynman integral, these can be treated by an algorithmic
procedure developed by Bogner and Weinzierl in \cite{BoWe2}
(see also the short survey \cite{BoWe3}). The algorithm is designed
to split the divergent integral into sectors where a change of variable
that introduces a blowup at the origin isolates the divergence 
as a pole in a parameter $1/\epsilon$. One can then do a subtraction
of this polar part in the Laurent series expansion in the variable
$\epsilon$ and eliminate the divergence. The iteration part of
the algorithm is based on Hironaka's polyhedral game and it
is shown in \cite{BoWe2} that the resulting algorithm terminates 
in finite time.

If one uses this approach in our context one will have to show that
the changes of variables introduced in the process of evaluating
the integrals in sectors do not alter the motivic nature of the loci
involved.

\subsection{Deformations}

An alternative to the use of blowups is the use of deformations.
We discuss here the simplest possible procedure one can think
of that uses deformations of the graph hypersurface (or of the
determinant hypersurface). It is not the most satisfactory deformation
method, because it does not lead immediately to a ``minimal
subtraction" procedure, but it suffices here to illustrate the idea.

Consider the original parametric Feynman integral of the form
\begin{equation}\label{Intalphabeta}
 \int_{\sigma_n} \frac{P_\Gamma(t,p)^\beta \omega_n}{\Psi_\Gamma(t)^\alpha}, 
\end{equation}
with exponents $\alpha$ and $\beta$ as in \eqref{paramInt},
$$ \alpha = -n +(\ell+1)D/2, \ \ \ \beta=-n+D\ell/2. $$
Again, for our purposes, we can assume to work in the ``stable range"
where $D$ is sufficiently large so that both $\alpha$ and $\beta$ are
positive. The case of small $D$, which is of direct physics interest,
leads one to the different problem of considering the hypersurfaces
defined by $P_\Gamma(t,p)$, as a function of the external momenta $p$
and the singularities produced by the intersections of these with the
domain of integration. This type of analysis can be found in the
physics literature, for instance in \cite{Todorov}. See also 
\cite{BjDr}, \S 18. 

Assuming to work in the range where $\alpha$ and $\beta$ are
positive, one can choose to regularize the integral \eqref{Intalphabeta}
by introducing a deformation parameter $\epsilon \in \C \smallsetminus \R_+$
and replaing \eqref{Intalphabeta} with the deformed
\begin{equation}\label{Intdeformed}
\int_{\sigma_n} \frac{P_\Gamma(t,p)^\beta \omega_n}{(\Psi_\Gamma(t)-\epsilon)^\alpha}.
\end{equation}
This has the effect of replacing, as locus of the singularities of the integrand, the
graph hypersurface $\hat X_\Gamma =\{ \Psi_\Gamma(t) =0 \}$, with the level
set $\hat X_{\Gamma,\epsilon}=\{ \Psi_\Gamma(t)=\epsilon \}$ of the map
$\Psi_\Gamma :\A^n \to \A$. For a choice of $\epsilon$ in the cut plane $\C \smallsetminus \R_+$,
the hypersurface $\hat X_{\Gamma,\epsilon}$ does not intersect the domain of integration
$\sigma_n$. In fact, for $t_i \geq 0$ one has $\Psi_\Gamma(t)\geq 0$. This choice has
therefore the effect of desingularizing the integral. The resulting function of $\epsilon$ extends
holomorphically to a function on $\C \smallsetminus I$, where $I\subset \R_+$ is the bounded 
interval of values of $\Psi_\Gamma$ on $\sigma_n$.

When we transform the parametric integral using the map $\tau_\Gamma$ into
an integral of a form defined on the complement of the determinant hypersurface
$\hat \cD_\ell$ in $\A^{\ell^2}$ on a domain of integration $\tau_\Gamma(\sigma_n)$
with boundary on the divisor $\hat\Sigma_{\ell, g}$, we can similarly separate the
divisor from the hypersurface by the same deformation, where instead of the
locus $\hat\cD_\ell =\{ \det(x)=0 \}$ one considers the level set $\hat\cD_{\ell,\epsilon}
=\{ \det(x) = \epsilon \}$,  so that $\hat\cD_{\ell,\epsilon}$ does not intersect 
$\tau_\Gamma(\sigma_n)$. The nature of the period described by the deformed 
integral is then controlled by the motive $\m(X_\epsilon,Y_\epsilon)$ for 
$X_\epsilon=\A^{\ell^2}\smallsetminus \hat\cD_{\ell,\epsilon}$ and $Y_\epsilon=\hat\Sigma_{\ell,g}
\smallsetminus (\hat\cD_{\ell,\epsilon}\cap \hat\Sigma_{\ell,g})$.
The question becomes then whether the motivic nature of $\m(X,Y)$ with $X=X_0$
and $Y=Y_0$ and $\m(X_\epsilon,Y_\epsilon)$ is the same. This in general is not
the case, as one can easily construct examples of fibrations where the generic fiber 
is not a mixed Tate motive while the special one is.
However, in this setting one is dealing with a very special case, where the 
deformed variety $\hat\cD_{\ell,\epsilon}$ is given by matrices of fixed determinant.
Up to a rescaling, one can check that the fiber $\hat\cD_{\ell,1}=\SL_n$ is indeed
a mixed Tate motive, from the general results of Biglari \cite{Biglari1}, \cite{Biglari2}
on reductive groups. Thus, over a set of algebraic values of $\epsilon$ one does 
not leave the world of mixed Tate motives.  This will give a statement on the
nature of the regularized Feynman integrals as a period of a mixed Tate motive
$\m(X_\epsilon,Y_\epsilon)$ and reduces then the problem to 
that of removing the divergence as $\epsilon \to 0$, in such a way that what
remains is a convergent integral whose nature as a period is controlled by the
original motive $\m(X,Y)$.

A different approach to the regularization of parametric Feynman integrals using
deformations was discussed in \cite{Mar} in terms of Leray cocycles and
a related regularization procedure. 

\bigskip

{\bf Acknowledgment. } The first author is partially supported by NSA grant H98230-07-1-0024. 
The second author is partially supported by NSF grant DMS-0651925. This work was partly
carried out during stays of the authors at the MPI and MSRI. The first author also thanks 
the California Institute of Technology, where part of this work was done.

\end{document}